\documentclass[twoside]{article} 

\usepackage[accepted]{aistats2017}

\pdfoutput=1

\usepackage{hyperref}
\hypersetup{urlcolor=rust}
\hypersetup{citecolor=dkblue}
\hypersetup{linkcolor=dkblue}

\usepackage[numbers,sort&compress,square]{natbib}

\usepackage{amsmath,amsbsy,amsgen,amscd,amssymb,amsthm,bm}
\usepackage{mathtools}

\usepackage[scaled=1.2]{urwchancal}

\usepackage{multicol}

\usepackage[usenames,dvipsnames]{xcolor}
\usepackage{graphicx}
\usepackage{tikz}
\graphicspath{{art/}}

\definecolor{dark-gray}{gray}{0.3}
\definecolor{dkgray}{rgb}{.4,.4,.4}
\definecolor{dkblue}{rgb}{0,0,.5}
\definecolor{medblue}{rgb}{0,0,.75}
\definecolor{rust}{rgb}{0.5,0.1,0.1}
\definecolor{purple}{rgb}{0.3,0.0,.4}

\usepackage{enumitem}
\usepackage{float}
\usepackage{caption}
\usepackage{subcaption}
\usepackage[font=small,margin=0.25in,labelfont={sc},labelsep={colon}]{caption}

\usepackage[noend]{algpseudocode}
\usepackage{algorithm,algorithmicx}

\algrenewcommand\alglinenumber[1]{\sf\tiny\color{medblue}{#1}\quad}
\algrenewcommand\algorithmicrequire{\textbf{Input:}}
\algrenewcommand\algorithmicensure{\textbf{Output:}}

\newtheorem{theorem}{Theorem}

\newtheorem{fact}[theorem]{Fact}

\theoremstyle{definition}

\newcommand{\myparagraph}[1]{\textbf{\textcolor{black}{#1}}}

\newcommand{\R}{\mathbb{R}}
\newcommand{\C}{\mathbb{C}}

\newcommand{\eps}{\varepsilon}
\newcommand{\econst}{\mathrm{e}}

\newcommand{\vct}[1]{\bm{#1}}
\newcommand{\mtx}[1]{\bm{#1}}

\newcommand{\diag}{\operatorname{diag}}
\newcommand{\trace}{\operatorname{tr}}
\newcommand{\rank}{\operatorname{rank}}

\newcommand{\dist}{\operatorname{dist}}
\newcommand{\abs}[1]{\left\vert #1 \right\vert}
\newcommand{\norm}[1]{\left\Vert #1 \right\Vert}
\newcommand{\fnorm}[1]{\norm{#1}_{\mathrm{F}}}

\newcommand{\ip}[2]{\left\langle #1, \, #2 \right\rangle}

\newcommand{\psdge}{\succcurlyeq}

\newcommand{\Expect}{\operatorname{\mathbb{E}}}

\newcommand{\minimize}{\text{minimize}}
\newcommand{\subjto}{\text{s.t.}}

\newcommand{\grad}{\nabla}
\newcommand{\huber}{\operatorname{huber}}

\def\istwocolumn{1}
\renewcommand{\Comment}[1]{}

\def\hassupplement{1}

\begin{document}

\twocolumn

\runningtitle{Convex Low-Rank Matrix Optimization with Optimal Storage}

\runningauthor{Yurtsever, Udell, Tropp, Cevher}

\twocolumn[

\aistatstitle{Sketchy Decisions: Convex Low-Rank Matrix Optimization \\ with Optimal Storage
\ifdefined \hassupplement \\
{\small Including supplementary appendix} \fi}

\ifdefined \isanonymous
\aistatsauthor{ Anonymous Authors }
\aistatsaddress{ Anonymous institutions }
\else
\aistatsauthor{ Alp Yurtsever \And Madeleine Udell \And Joel A.~Tropp \And Volkan Cevher}
\aistatsaddress{ EPFL \And Cornell \And Caltech \And EPFL }
\fi

]

\begin{abstract}
This paper concerns a fundamental class of convex matrix optimization problems.
It presents the first algorithm that uses optimal storage 
and provably computes a low-rank approximation of a solution.
In particular, when all solutions have low rank, the algorithm
converges to a solution.
This algorithm, SketchyCGM, modifies a standard convex optimization scheme,
the conditional gradient method, to store only a small
randomized sketch of the matrix variable.
After the optimization terminates,
the algorithm extracts a low-rank approximation of the solution
from the sketch.  
In contrast to nonconvex heuristics,
the guarantees for SketchyCGM do not rely on statistical
models for the problem data.
Numerical work 
demonstrates the benefits of SketchyCGM over
heuristics.
\end{abstract}

\section{MOTIVATION}

This paper discusses a fundamental class of convex
matrix optimization problems with low-rank solutions.
We argue that the main obstacle that prevents us from solving these
problems at scale is not arithmetic, but storage.
We exhibit the first provably correct algorithm
for these problems
with optimal storage.

\subsection{Vignette: Matrix Completion}

To explain the challenge, 
we consider the problem of low-rank matrix completion.

Let $\mtx{X}_{\natural} \in \R^{m\times n}$ be an unknown matrix,
but assume that a bound $r$ on the rank of $\mtx{X}_{\natural}$ is available,
where $r \ll \min\{m, n\}$.
Suppose that we record noisy observations of a subset $E$
of entries from the matrix:
$$
b_{ij} = (\mtx{X}_{\natural})_{ij} + \xi_{ij}
\quad\text{for $(i, j) \in E$.}
$$
The variables $\xi_{ij} \in \R$ model (unknown) noise.
The goal is to approximate the full matrix $\mtx{X}_{\natural}$.

Matrix completion arises in machine learning applications,
such as recommendation systems~\citep{SRJ04:Maximum-Margin-Matrix}.

We can frame the matrix completion problem 
as a rank-constrained optimization:
\begin{equation} \label{eqn:matrix-completion-rank}
\underset{\mtx{X} \in \R^{m \times n}}{\minimize} \
\sum\limits_{(i,j)\in E} (x_{ij} - b_{ij})^2
\quad \subjto\quad
\rank \mtx{X} \leq r.
\end{equation}
In general, the formulation~\eqref{eqn:matrix-completion-rank} 
is intractable.  Instead, we retrench
to a tractable convex problem \citep{SRJ04:Maximum-Margin-Matrix,CP10:Matrix-Completion}: 
\begin{equation} \label{eqn:matrix-completion}
\underset{\mtx{X} \in \R^{m \times n}}{\minimize}\
\sum\limits_{(i, j) \in E} ( x_{ij} - b_{ij} )^2
\quad\subjto\quad
\norm{\mtx{X}}_{S_1} \leq \alpha.
\end{equation}
The Schatten $1$-norm $\norm{\cdot}_{S_1}$ returns the sum of the
singular values of its argument; it is an effective proxy for the rank~\citep{Faz02:Matrix-Rank}.
Adjusting the value of the parameter $\alpha$ modulates
the rank of a solution $\mtx{X}_{\star}$ of~\eqref{eqn:matrix-completion}.
If we have enough data and choose $\alpha$ well, we expect that each solution
$\mtx{X}_{\star}$ approximates the target matrix $\mtx{X}_{\natural}$.

The convex problem~\eqref{eqn:matrix-completion}
is often a good model for matrix completion when the number of observations
$\abs{E} = \tilde{\mathcal{O}}(r(m+n))$, where $\tilde{\mathcal{O}}$ suppresses log-like factors;
see~\citep{SRJ04:Maximum-Margin-Matrix,CP10:Matrix-Completion}. 
We can write a rank-$r$ approximation to a solution $\mtx{X}_{\star}$
using $\Theta(r(m+n))$ parameters.  
Thus, we can express the problem and an approximate solution
with $\tilde{\mathcal{O}}(r(m+n))$ storage.

Nevertheless, we need fully $mn$ numbers to express the decision
variable $\mtx{X}$ for the optimization problem~\eqref{eqn:matrix-completion}.
The cost of storing the decision variable
prevents us from solving large-scale instances
of~\eqref{eqn:matrix-completion},
even without worrying about arithmetic.

This discrepancy raises a question:
\textbf{Is there an algorithm that computes an approximate solution
to~\eqref{eqn:matrix-completion} using the optimal storage $\tilde{\mathcal{O}}(r (m+n))$?}

\subsection{Vignette: Phase Retrieval}
\label{sec:phase-vignette}
Here is another instance of the same predicament.

Fix a vector $\vct{x}_\natural \in \mathbb{C}^n$.
Suppose that we acquire $d$ noisy quadratic measurements of $\vct{x}_\natural$
with the form
\begin{equation} \label{eqn:phase-retrieval-meas}
b_i = \abs{ \ip{\vct{a}_i}{ \smash{\vct{x}_\natural}} }^2 + \xi_i
\quad\text{for $i = 1, 2, \dots, d$}.
\end{equation}
The $\vct{a}_i \in \C^n$ are known measurement vectors,
and the $\xi_i \in \R$ model measurement noise.
Given the data $\vct{b}$ and the vectors $\vct{a}_i$,
the phase retrieval problem asks us to reconstruct $\vct{x}_\natural$ up to a global phase shift.

Phase retrieval problems are prevalent in imaging science because
it is easier to measure the intensity of light than its phase.
In practice, the vectors $\vct{a}_i$ are structured because
they reflect the physics of the imaging system.
\ifdefined \hassupplement {
For more details, see Appendix~\ref{app:PR} and the examples in~\citep{BBCE09:Painless-Reconstruction,CMP11:Array-Imaging,CESV13:Phase-Retrieval,HCO+15:Solving-Ptychography}. }
\else {
See the supplement 
and~\citep{BBCE09:Painless-Reconstruction,CMP11:Array-Imaging,CESV13:Phase-Retrieval,HCO+15:Solving-Ptychography}.
} \fi

Let us outline a convex approach~\citep{BBCE09:Painless-Reconstruction,CMP11:Array-Imaging,
CESV13:Phase-Retrieval,HCO+15:Solving-Ptychography} to the phase retrieval problem.
The data~\eqref{eqn:phase-retrieval-meas} satisfies
$$
b_i = \vct{a}_i^* \mtx{X}_{\natural} \vct{a}_i + \xi_i
\quad\text{where}\quad
\mtx{X}_{\natural} = \vct{x}_{\natural} \vct{x}_{\natural}^*.
$$
Thus, we can formulate phase retrieval as
\begin{equation} \label{eqn:phase-rank-min}
\begin{aligned}
&\underset{\mtx{X} \in \C^{n \times n}}{\minimize} && \sum\nolimits_{i=1}^d (\vct{a}_i^* \mtx{X} \vct{a}_i - b_i)^2 \\
&\subjto &&\rank \mtx{X} = 1, \quad \mtx{X} \psdge \mtx{0}.
\end{aligned}
\end{equation}
Now, pass to the convex problem
\begin{equation}  \label{eqn:phase-convex}
\begin{aligned}
&\underset{\mtx{X}  \in \C^{n \times n}}{\minimize} && \sum\nolimits_{i=1}^d (\vct{a}_i^* \mtx{X} \vct{a}_i - b_i)^2 \\
&\subjto &&\trace \mtx{X} \leq \alpha, \quad \mtx{X} \psdge \mtx{0}.
\end{aligned}
\end{equation}
We can estimate the parameter $\alpha \in \R_+$ from $\vct{a}_i$ and $\vct{b}$;
see~\cite[Sec.~II]{YHC15:Scalable-Convex}.  To approximate the
true vector $\vct{x}_{\natural}$, we compute a top eigenvector $\vct{x}_{\star}$
of a solution to~\eqref{eqn:phase-convex}.

This procedure is often an effective approach for phase retrieval
when the number of measurements $d = \Theta(n)$;
see \cite[Sec.~2.8]{Tro15:Convex-Recovery}.
Once again, we recognize a discrepancy. 
The problem data $\vct{b}\in \R^d$
and the approximate solution $\vct{x}_{\star} \in \C^n$
use storage $\Theta(n)$, but the matrix variable
in~\eqref{eqn:phase-convex} requires $\Theta(n^2)$ storage.

We may ask:
\textbf{Is there an algorithm that computes an approximate solution to~\eqref{eqn:phase-convex}
using the optimal storage $\Theta(n)$?}

\subsection{Low-Rank Matrix Optimization Methods}

Matrix completion and phase retrieval are examples of
\emph{convex low-rank matrix optimization} (CLRO) problems.
Informally, this class contains convex optimization problems
whose decision variable is a matrix 
and whose solutions are (close to) low rank.
These problems often arise as convex relaxations of 
rank-constrained problems; however, the convex formulations
are important in their own right.

There has been extensive empirical and theoretical work to
validate the use of CLROs in a spectrum of applications.
For example, see~\citep{GW95:Improved-Approximation,Faz02:Matrix-Rank, 
CP10:Matrix-Completion, 
CESV13:Phase-Retrieval,HCO+15:Solving-Ptychography}.

Over the last 20 years, optimization researchers have developed
a diverse collection of algorithms for CLRO problems.
Surprisingly, every extant method lacks guarantees on storage or convergence (or both).

Convex optimization algorithms dominated the early literature on
algorithms for CLRO.  The initial efforts, such as~\citep{HRVW96:Interior-Point-Method},
focused on interior-point methods, whose storage and arithmetic costs are forbidding.
To resolve this issue, researchers turned to
first-order convex algorithms,
including bundle methods~\citep{HR00:Spectral-Bundle},
(accelerated) proximal gradient
methods~\citep{Roc76:Monotone-Operators,AT06:Interior-Gradient,TY10:Accelerated},
and the conditional gradient method (CGM)
\citep{FW56:Algorithm-Quadratic,LP66:Minimization-Methods,
Haz08:Sparse-Approximate,Cla10:Coresets-Sparse,Jag13:Revisiting-Frank-Wolfe}.

Convex algorithms are guaranteed to solve a CLRO.
They come with a complete theory, including rigorous
stopping criteria and bounds on convergence rates.
They enjoy robust performance in practice.
On the other hand, convex algorithms from the literature
do not scale well enough to solve large CLRO problems
because they operate on and store full-size matrices.

The CGM iteration is sometimes touted as a low-storage method
for CLRO~\citep{Jag13:Revisiting-Frank-Wolfe}.
Indeed, CGM is guaranteed to increase the rank of an iterate
by at most one per iteration.
Nevertheless, the algorithm converges slowly, so intermediate
iterates can have very high rank.
CGM variants, such as~\citep{RSW15:Conditional-Gradient,YHC15:Scalable-Convex},
that control the rank of iterates lack storage guarantees or
may not converge to a global optimum.

Recently, many investigators have sought recourse in
nonconvex heuristics for solving CLROs. 
This line of work depends on the factorization
idea of Burer \& Monteiro~\citep{BM03:Nonlinear-Programming},
which rewrites the matrix variable as a product of two low-rank
factors.  There are many heuristic procedures,
e.g.,~\citep{BM03:Nonlinear-Programming,JNS13:Low-Rank-Matrix,BKS16:Dropping-Convexity,BVB16:Nonconvex-Burer-Monteiro},
that use clever initialization and nonlinear
programming schemes in an attempt to optimize the factors directly.
The resulting algorithms can have optimal storage costs, and
they may achieve a fast rate of local convergence.

There has been an intensive effort to justify the
application of nonconvex heuristics for CLRO.
To do so, researchers often frame unverifiable statistical
assumptions on the problem data.  For example,
in the matrix completion problem~\eqref{eqn:matrix-completion},
it is common to assume that the entries of the matrix are revealed according
to some ideal probability distribution~\cite{JNS13:Low-Rank-Matrix,
CP10:Matrix-Completion}. 
When these assumptions fail, nonconvex heuristics
can converge to the wrong point, or they may
even diverge.

\myparagraph{Contributions.}
This paper explains how to extend the convex optimization algorithm CGM
to obtain an approximate solution to a class of CLRO problems using optimal storage.
Our algorithm operates much like CGM,
but it never forms the matrix variable explicitly.
Instead, we maintain a small randomized sketch of the matrix variable over the course of the iteration
by means of a bespoke sketching method~\citep{TYUC17:Randomized-Single-View}.
After the optimization method converges, we extract an approximate solution from the sketch.
This technique achieves optimal storage,
yet it converges under the same conditions
and with the same guarantees as CGM.

In summary, this paper presents a solution to the problems posed above: 
{the first algorithm for convex low-rank matrix optimization problems
that provably uses optimal storage to compute an approximate solution.}

\subsection{Notation}

We write $\norm{ \cdot }$ for the Euclidean norm, $\fnorm{ \cdot }$
for the Frobenius norm, and $\norm{\cdot}_{S_1}$ for the Schatten 1-norm
(aka the \emph{trace norm} or the \emph{nuclear norm}).
Depending on context, $\ip{ \cdot }{ \cdot }$
refers to the Euclidean or Frobenius inner product. 
The symbol ${}^*$ denotes the conjugate transpose of a vector or matrix, as well as the
adjoint of a linear map.  The dagger ${}^\dagger$ refers to the pseudoinverse.
The symbol $[\mtx{M}]_r$ stands for a best rank-$r$
Frobenius-norm approximation of the matrix $\mtx{M}$.
The function $\dist_{\rm F}(\mtx{M}; S)$ returns the minimum
Frobenius-norm distance from $\mtx{M}$ to a set $S$.
The symbol $\psdge$ denotes the semidefinite order.
We use the computer science interpretation of the order notation 
$\mathcal{O}, \tilde{\mathcal{O}}, \Omega, \Theta$.

\section{A LOW-RANK MATRIX OPTIMIZATION PROBLEM} 

Let us begin with a generalization of the convex matrix completion
formulation~\eqref{eqn:matrix-completion}.  In \S\ref{sec:CGM-psd},
we return to the psd setting of the phase retrieval problem~\eqref{eqn:phase-convex}.

We consider a convex program with a matrix variable:
\begin{equation} \label{eqn:model-prob}
\underset{\mtx{X} \in \R^{m \times n}}{\minimize} \ f( \mathcal{A}\mtx{X} )
\quad\subjto\quad \norm{ \mtx{X} }_{S_1} \leq \alpha.
\end{equation}
\noindent
The linear operator $\mathcal{A} : \R^{m \times n} \to \R^d$ and its adjoint
$\mathcal{A}^* : \R^d \to \R^{m \times n}$ take the form
\ifdefined \istwocolumn
\begin{equation} \label{eqn:linear-maps}
\begin{aligned}
\mathcal{A}\mtx{X}& = \begin{bmatrix}
	\ip{ \mtx{A}_1 }{ \mtx{X} } & \dots & \ip{ \mtx{A}_d }{ \mtx{X} }
	\end{bmatrix}; \\
\mathcal{A}^*\vct{z} &= \sum\nolimits_{i=1}^d z_i \mtx{A}_i.
\end{aligned}
\end{equation}
\else
\begin{equation} \label{eqn:linear-maps}
\mathcal{A}\mtx{X} = \begin{bmatrix}
	\ip{ \mtx{A}_1 }{ \mtx{X} } & \dots & \ip{ \mtx{A}_d }{ \mtx{X} }
	\end{bmatrix}^*
\quad\text{and}\quad
\mathcal{A}^*\vct{z} = \sum\nolimits_{i=1}^d z_i \mtx{A}_i.
\end{equation}
\fi
Each coefficient matrix $\mtx{A}_i \in \R^{m \times n}$.

We interpret $\mathcal{A}\mtx{X}$ as a set of linear measurements
of the matrix $\mtx{X}$.  For example, in the matrix completion problem~\eqref{eqn:matrix-completion},
the image $\mathcal{A}\mtx{X}$ lists the entries of $\mtx{X}$ indexed by the set $E$.

The function $f : \R^d \to \R$ is convex and continuously differentiable.
In many situations, it is natural to
regard the objective function as a loss:
$f(\mathcal{A}\mtx{X}) = \mathrm{loss}(\mathcal{A}\mtx{X}; \vct{b})$
for a vector $\vct{b} \in \R^d$ of measured data.

By choosing the parameter $\alpha \in \R_+$ to be sufficiently small,
we can often ensure that each minimizer of~\eqref{eqn:model-prob} is
low-rank or close to low-rank.

Our goal is to develop a practical algorithm that provably computes
a low-rank approximation of a solution to the problem~\eqref{eqn:model-prob}.

To validate \eqref{eqn:model-prob} as a model for a given application,
one must undertake a separate empirical or theoretical study.
We do not engage this question in our work.

\subsection{Storage Issues}

Suppose that we want to produce a low-rank approximation
to a solution of a generic instance of the problem~\eqref{eqn:model-prob}.
What kind of storage can we hope to achieve?

It is clear that we need $\Theta(r(m+n))$ numbers to
express a rank-$r$ approximate solution to \eqref{eqn:model-prob}.
We must also understand how much extra storage is incurred
because of the specific problem instance $(\mathcal{A}, f)$.

It is natural to instate a \emph{black-box model}
for the linear map $\mathcal{A}$, its adjoint $\mathcal{A}^*$,
and the objective function $f$.
For arbitrary vectors $\vct{u} \in \R^m$ and $\vct{v} \in \R^n$ and $\vct{z} \in \R^d$,
assume we have routines that can compute
\begin{equation} \label{eqn:A-action}
\mathcal{A}(\vct{uv}^*)
\quad\text{and}\quad
\vct{u}^*(\mathcal{A}^* \vct{z})
\quad\text{and}\quad
(\mathcal{A}^* \vct{z}) \vct{v}.
\end{equation}
We also assume routines for evaluating the function $f$
and its gradient $\grad f$ for any argument $\vct{z} \in \R^d$.
We may neglect the storage used to compute these primitives.
Every algorithm based on these primitives must use storage
$\Omega(m + n + d)$ just to represent their outputs.

Thus, under the black-box model, any algorithm that produces a rank-$r$ solution to a generic
instance of~\eqref{eqn:model-prob} must use storage $\Omega(d + r(m+n))$.
We say that an algorithm is \emph{storage optimal} if it achieves this bound.

The parameter $d$ often reflects the amount of data that we have acquired,
and it is usually far smaller than the dimension $mn$ of the matrix variable
in~\eqref{eqn:model-prob}.

The problems that concern us are data-limited; that is, $d \ll mn$.
This is the situation where a strong structural prior
(e.g., low rank or small Schatten 1-norm)
is essential for fitting the data.  This challenge is common
in machine learning problems (e.g., matrix completion for recommendation systems),
as well as in scientific applications (e.g., phase retrieval).

To the best of our knowledge, no extant algorithm for~\eqref{eqn:model-prob}
is guaranteed to produce an approximation of an optimal point and
also enjoys optimal storage cost.

\section{CONDITIONAL GRADIENT} 

To develop our algorithm for the model problem~\eqref{eqn:model-prob},
we must first describe a standard algorithm called the~\emph{conditional gradient method} (CGM).
Classic and contemporary references include~\citep{FW56:Algorithm-Quadratic,LP66:Minimization-Methods,
Haz08:Sparse-Approximate,Cla10:Coresets-Sparse,Jag13:Revisiting-Frank-Wolfe}.

\ifdefined \islong
\subsection{Intuition.}

CGM is a variant of gradient descent that replaces the
gradient by a certain approximation.
Form the best linear underestimate of the objective function $f \circ \mathcal{A}$
at a fixed matrix $\mtx{X}$:
\ifdefined \istwocolumn
\begin{equation} \label{eqn:CGM-model}
f(\mathcal{A}\mtx{H})
	\geq f(\mathcal{A}\mtx{X})
	+ \ip{ \mtx{H} - \mtx{X} }{ \mathcal{A}^* (\grad f (\mathcal{A}\mtx{X}))}
\end{equation}
for $\norm{\mtx{H}}_{S_1} \leq \alpha$.
\else
\begin{equation} \label{eqn:CGM-model}
f(\mathcal{A}\mtx{H})
	\geq f(\mathcal{A}\mtx{X})
	+ \ip{ \mtx{H} - \mtx{X} }{ \mathcal{A}^* (\grad f (\mathcal{A}\mtx{X}))}
	\quad\text{for}\quad \norm{\mtx{H}}_{S_1} \leq \alpha.
\end{equation}
\fi
To improve an estimated solution $\mtx{X}$ of~\eqref{eqn:model-prob},
CGM finds an update direction $\mtx{H}_+$ by minimizing the right-hand
side of~\eqref{eqn:CGM-model} over 
$\norm{\mtx{H}}_{S_1} \leq \alpha$. 
The algorithm takes a small step in the direction $\mtx{H}_+ - \mtx{X}$.
We can interpret the matrix $\mtx{H}_+ - \mtx{X}$ as an approximate negative gradient
of $f \circ \mathcal{A}$ at $\mtx{X}$.

\fi

\subsection{The CGM Iteration}
Here is the CGM algorithm for~\eqref{eqn:model-prob}.
Start with a feasible point, such as
\begin{equation} \label{eqn:CGM-primal-init}
\mtx{X}_0 = \mtx{0} \in \R^{m \times n}.
\end{equation}
At each iteration $t = 0, 1, 2, \dots$, compute an update direction $\mtx{H}_t$
using the formulas
\ifdefined \istwocolumn
\begin{equation} \label{eqn:CGM-primal-dir}
\begin{aligned}
(\vct{u}_t, \vct{v}_t) & = \texttt{MaxSingVec}(\mathcal{A}^*(\grad f(\mathcal{A}\mtx{X}_t))); \\
\mtx{H}_t & = - \alpha \vct{u}_t \vct{v}_t^*.
\end{aligned}
\end{equation}
\else
\begin{equation} \label{eqn:CGM-primal-dir}
\mtx{H}_t = - \alpha \, \vct{u}_t \vct{v}_t^*
\quad\text{where}\quad
(\vct{u}_t, \vct{v}_t) =
\texttt{MaxSingVec}(\mathcal{A}^*(\grad f(\mathcal{A}\mtx{X}_t))).
\end{equation}
\fi
$\texttt{MaxSingVec}$ returns a left/right pair of maximum singular vectors.
Update the decision variable: 
\ifdefined \istwocolumn
\begin{equation} \label{eqn:CGM-primal-update}
\mtx{X}_{t+1} = (1 - \eta_t) \mtx{X}_t + \eta_t \mtx{H}_t
\end{equation}
where $\eta_t = 2/(t+2)$.
\else
\begin{equation} \label{eqn:CGM-primal-update}
\mtx{X}_{t+1} = (1 - \eta_t) \mtx{X}_t + \eta_t \mtx{H}_t
\quad\text{where}\quad
\eta_t = 2/(t+2).
\end{equation}
\fi
The convex combination~\eqref{eqn:CGM-primal-update}
remains feasible for~\eqref{eqn:model-prob}
because $\mtx{X}_t$ and $\mtx{H}_t$ are feasible.

CGM is a valuable algorithm for~\eqref{eqn:model-prob}
because we can efficiently find the rank-one update direction $\mtx{H}_t$
by means of the singular vector computation~\eqref{eqn:CGM-primal-dir}. 
The weak point of CGM is that the rank of $\mtx{X}_t$
typically increases with $t$,
and the peak rank of an iterate $\mtx{X}_t$ is often much
larger than the rank of the solution of~\eqref{eqn:model-prob}.
\ifdefined \hassupplement {
See Figures~\ref{fig:evolution-spectrum} and~\ref{fig:fourier-ptychography-convergence}
for an illustration.
} \fi

\subsection{The CGM Stopping Rule}
The CGM algorithm admits a simple stopping criterion.
Given a suboptimality parameter $\eps > 0$, we halt the CGM iteration when
the duality gap $\delta_t \leq \eps$:
\begin{equation} \label{eqn:CGM-primal-stop}
\delta_t = \ip{ \mathcal{A}\mtx{X}_t - \mathcal{A}\mtx{H}_t }{ \grad f(\mathcal{A}\mtx{X}_t) }
	\leq \eps.
\end{equation}
Let $\mtx{X}_{\star}$ be an optimal point for~\eqref{eqn:model-prob}.
It is not hard to show~\cite[Sec.~2]{Jag13:Revisiting-Frank-Wolfe} that
\begin{equation} \label{eqn:CGM-gap}
f(\mathcal{A}\mtx{X}_{t}) - f(\mathcal{A}\mtx{X}_{\star})
\leq \delta_t.
\end{equation}
Thus, the condition~\eqref{eqn:CGM-primal-stop} ensures that the objective value
$f(\mathcal{A}\mtx{X}_t)$ is $\eps$-suboptimal.
The CGM iterates satisfy \eqref{eqn:CGM-primal-stop} 
within $\mathcal{O}(\eps^{-1})$ iterations~\cite[Thm.~1]{Jag13:Revisiting-Frank-Wolfe}.

\subsection{The Opportunity}

The CGM iteration~\eqref{eqn:CGM-primal-init}--\eqref{eqn:CGM-primal-update}
requires $\Theta( mn )$ storage
because it maintains the $m \times n$ matrix decision variable $\mtx{X}_t$.
We develop a remarkable extension of CGM that
provably computes a rank-$r$ approximate solution to~\eqref{eqn:model-prob}
with working storage $\Theta( d + r(m+n) )$.
Our approach depends on two efficiencies:

\begin{itemize}
\item	We use the low-dimensional ``dual'' variable $\vct{z}_t = \mathcal{A}\mtx{X}_t \in \R^d$ to drive the iteration.

\item	Instead of storing $\mtx{X}_t$, we maintain a small randomized sketch with size $\Theta( r (m + n) )$.
\end{itemize}

\noindent
It is easy to express the CGM iteration in terms of the ``dual'' variable $\vct{z}_t = \mathcal{A}\mtx{X}_t$.
We can obviously rewrite the formula~\eqref{eqn:CGM-primal-dir}
for computing the rank-one update direction $\mtx{H}_t$ in terms of $\vct{z}_t$.
We obtain an update rule for $\vct{z}_t$ by applying the linear
map $\mathcal{A}$ to~\eqref{eqn:CGM-primal-update}.
Likewise, the stopping criterion~\eqref{eqn:CGM-primal-stop}
can be evaluated using $\vct{z}_t$ and $\mtx{H}_t$.
Under the black-box model~\eqref{eqn:A-action},
the dual formulation of CGM has storage cost $\Theta(m + n + d)$. 

Yet the dual formulation has a flaw:
it ``solves'' the problem~\eqref{eqn:model-prob},
but we do not know the solution! 

Indeed, we must also track
the evolution~\eqref{eqn:CGM-primal-update} of the primal decision variable $\mtx{X}_t$.
In the next subsection, we summarize a randomized sketching method~\cite{TYUC17:Randomized-Single-View}
that allows us to compute an accurate rank-$r$ approximation
of $\mtx{X}_t$ but operates with storage $\Theta(r(m+n))$.

\section{MATRIX SKETCHING} 
\label{sec:sketch}

Suppose that $\mtx{X} \in \R^{m \times n}$ is a matrix
that is presented to us as a stream of 
linear updates,
as in~\eqref{eqn:CGM-primal-update}.
For a parameter $r \ll \min\{m, n\}$, we wish to maintain
a small sketch that allows us to compute a rank-$r$
approximation of the final matrix $\mtx{X}$.
\ifdefined \isanonymous
Let us summarize an approach from the recent paper~\cite{TYUC17:Randomized-Single-View}.
\else
Let us summarize an approach developed in our paper~\cite{TYUC17:Randomized-Single-View}.
\fi

\subsection{The Randomized Sketch}

Draw and fix two independent standard normal matrices
$\mtx{\Omega}$ and $\mtx{\Psi}$ where
\ifdefined \istwocolumn
\begin{equation} \label{eqn:test-matrices}
\begin{array}{lll}
\mtx{\Omega} \in \R^{n \times k}
&\text{with}&
k = 2r + 1; \\
\mtx{\Psi} \in \R^{\ell \times m}
&\text{with}&
\ell = 4r + 3.
\end{array}
\end{equation}
\else
\begin{equation} \label{eqn:test-matrices}
\mtx{\Omega} \in \R^{n \times k}
\quad\text{with}\quad
k = 2r + 1
\quad\text{and}\quad
\mtx{\Psi} \in \R^{\ell \times m}
\quad\text{with}\quad
\ell = 4r + 3.
\end{equation}
\fi
The sketch consists of two matrices $\mtx{Y}$ and $\mtx{W}$
that capture the range and co-range of $\mtx{X}$:
\begin{equation} \label{eqn:sketch}
\mtx{Y} = \mtx{X}\mtx{\Omega} \in \R^{m \times k}
\quad\text{and}\quad
\mtx{W} = \mtx{\Psi} \mtx{X}  \in \R^{\ell \times n}.
\end{equation}
We can efficiently update the sketch $(\mtx{Y}, \mtx{W})$ 
to reflect a rank-one linear update to $\mtx{X}$ of the form
\begin{equation} \label{eqn:sketch-update}
\mtx{X} \gets \beta_1 \mtx{X} + \beta_2 \vct{u} \vct{v}^*.
\end{equation}
Both the storage cost for the sketch and the arithmetic cost of an update are $\Theta(r(m+n))$.

\subsection{The Reconstruction Algorithm}
The following procedure yields a rank-$r$ approximation
$\hat{\mtx{X}}$ of the matrix $\mtx{X}$ stored in the
sketch~\eqref{eqn:sketch}.
\ifdefined \istwocolumn
\begin{equation} \label{eqn:sketch-recon}
\mtx{Q} = \texttt{orth}(\mtx{Y}); \
\mtx{B} = (\mtx{\Psi} \mtx{Q})^{\dagger} \mtx{W}; \
\hat{\mtx{X}} = \mtx{Q} [\mtx{B}]_r.
\end{equation}
\else
\begin{equation} \label{eqn:sketch-recon}
\mtx{Y} = \mtx{QR}
\quad\text{and}\quad
\mtx{B} = (\mtx{\Psi} \mtx{Q})^{\dagger} \mtx{W}
\quad\text{and}\quad
\hat{\mtx{X}} = \mtx{Q} [\mtx{B}]_r.
\end{equation}
\fi
The matrix $\mtx{Q}$ 
has orthonormal columns that span the range of $\mtx{Y}$.
The extra storage costs of the reconstruction are negligible;
its arithmetic cost is $\Theta(r^2 (m + n))$.
See~\cite[\S4.2]{TYUC17:Randomized-Single-View} for the intuition
behind this method.  It achieves the following error bound.

\begin{theorem}[Reconstruction error \protect{\cite[Thm.~5.1]{TYUC17:Randomized-Single-View}}]
\label{thm:sketch-err-body}
Fix a target rank $r$.  Let $\mtx{X}$ be a matrix, and let $(\mtx{Y}, \mtx{W})$
be a sketch of $\mtx{X}$ of the form~\eqref{eqn:test-matrices}--\eqref{eqn:sketch}.
The procedure~\eqref{eqn:sketch-recon} yields a rank-$r$ matrix $\hat{\mtx{X}}$ with
$$
\Expect \fnorm{\smash{\mtx{X} - \hat{\mtx{X}}}}
	\leq 3\sqrt{2} \fnorm{\mtx{X} - [\mtx{X}]_r}.
$$
Similar bounds hold with high probability.
\end{theorem}

\myparagraph{Remarks.}
The sketch size parameters $(k, \ell)$ appearing in~\eqref{eqn:test-matrices}
are recommended to balance storage against reconstruction quality.
See~\cite{TYUC17:Randomized-Single-View} and our follow-up work
for more details and for other sketching methods.

\section{SKETCHING + CGM}

We are now prepared to present SketchyCGM, a storage-optimal extension
of the CGM algorithm for the convex problem~\eqref{eqn:model-prob}.
This method delivers a provably accurate low-rank approximation
to a solution of~\eqref{eqn:model-prob}.
See Algorithm~\ref{alg:sketch-CGM} for complete pseudocode.

\subsection{The SketchyCGM Iteration}
Fix the suboptimality $\eps$ and the rank $r$.
Draw and fix standard normal matrices $\mtx{\Omega} \in \R^{n \times k}$
and $\mtx{\Psi} \in \R^{\ell \times m}$ as in~\eqref{eqn:test-matrices}.
Initialize the iterate and the sketches:
\ifdefined \istwocolumn
\begin{equation} \label{eqn:CGM-sketch-init}
\vct{z}_0 = \vct{0}_d;
\
\mtx{Y}_0 = \mtx{0}_{m \times k};
\ \text{and}\
\mtx{W}_0 = \mtx{0}_{\ell \times n}.
\end{equation}
\else
\begin{equation} \label{eqn:CGM-sketch-init}
\vct{z}_0 = \vct{0} \in \R^d
\quad\text{and}\quad
\mtx{Y}_0 = \mtx{0} \in \R^{m \times k}
\quad\text{and}\quad
\mtx{W}_0 = \mtx{0} \in \R^{\ell \times n}.
\end{equation}
\fi
At each iteration $t = 0, 1, 2, \dots$,
compute an update direction via Lanczos or via randomized SVD~\cite{HMT11:Finding-Structure}:
\ifdefined \istwocolumn
\begin{equation} \label{eqn:CGM-sketch-dir}
\begin{aligned}
(\vct{u}_t, \vct{v}_t) & = \texttt{MaxSingVec}(\mathcal{A}^* (\grad f(\vct{z}_t))); \\
\vct{h}_t & = \mathcal{A}(-\alpha\vct{u}_t \vct{v}_t^*).
\end{aligned}
\end{equation}
\else
\begin{equation} \label{eqn:CGM-sketch-dir}
\vct{h}_t = \mathcal{A}(-\alpha\vct{u}_t \vct{v}_t^*)
\quad\text{where}\quad
(\vct{u}_t, \vct{v}_t) = \texttt{MaxSingVec}(\mathcal{A}^* (\grad f(\vct{z}_t))).
\end{equation}
\fi
Set the learning rate $\eta_t = 2/(t+2)$.
Update the iterate and the two sketches:
\begin{equation} \label{eqn:CGM-sketch-update}
\begin{aligned}
\vct{z}_{t+1} & = (1 - \eta_t) \vct{z}_t + \eta_t \vct{h}_t;  \\
\mtx{Y}_{t+1} &= (1 - \eta_t) \vct{Y}_t + \eta_t (-\alpha \vct{u}_t \vct{v}_t^*) \mtx{\Omega};  \\
\mtx{W}_{t+1} &= (1 - \eta_t) \vct{W}_t + \eta_t \mtx{\Psi} (-\alpha \vct{u}_t \vct{v}_t^*).
\end{aligned}
\end{equation}
The iteration continues until it triggers the stopping criterion:
\begin{equation} \label{eqn:CGM-sketch-stop}
\ip{ \vct{z}_t - \vct{h}_t }{ \grad f (\vct{z}_t) } \leq \eps.
\end{equation}
At any iteration $t$, we can form a rank-$r$ approximate solution
$\hat{\mtx{X}}_t$ of the model problem~\eqref{eqn:model-prob} 
by invoking the procedure~\eqref{eqn:sketch-recon} with $\mtx{Y} = \mtx{Y}_t$
and $\mtx{W} = \mtx{W}_t$.

\subsection{Guarantees}
Suppose that the CGM iteration~\eqref{eqn:CGM-primal-init}--\eqref{eqn:CGM-primal-update}
generates the sequence $(\mtx{X}_t : t=0,1,2,\dots)$ of decision variables
and the sequence $(\mtx{H}_t : t=0,1,2,\dots)$ of update directions.
It is easy to verify
that the SketchyCGM iteration~\eqref{eqn:CGM-sketch-init}--\eqref{eqn:CGM-sketch-update}
maintains the loop invariants
\ifdefined \istwocolumn
\begin{equation} \label{eqn:CGM-sketch-loop}
\begin{array}{lll}
\vct{z}_t = \mathcal{A}\vct{X}_t
&\text{and}&
\vct{h}_t = \mathcal{A}\vct{H}_t; \\
\mtx{Y}_t = \mtx{X}_t \mtx{\Omega}
&\text{and}&
\mtx{W}_t =  \mtx{\Psi} \mtx{X}_t.
\end{array}
\end{equation}
\else
\begin{equation} \label{eqn:CGM-sketch-loop}
\vct{z}_t = \mathcal{A}\vct{X}_t
\quad\text{and}\quad
\vct{h}_t = \mathcal{A}\vct{H}_t
\quad\text{and}\quad
\mtx{Y}_t = \mtx{X}_t \mtx{\Omega}
\quad\text{and}\quad
\mtx{W}_t =  \mtx{\Psi} \mtx{X}_t.
\end{equation}
\fi
In view of the inequality~\eqref{eqn:CGM-gap} and the invariant~\eqref{eqn:CGM-sketch-loop},
the stopping rule~\eqref{eqn:CGM-sketch-stop}
ensures that $\mtx{X}_t$ is an $\eps$-suboptimal solution
to~\eqref{eqn:model-prob} when the iteration halts.
Furthermore, Theorem~\ref{thm:sketch-err-body} ensures
that the computed solution $\hat{\mtx{X}}_t$ is a near-optimal rank-$r$
approximation of $\mtx{X}_t$ at each time $t$.

\subsection{Storage Costs}

The total storage cost is $\Theta(d + r(m+n))$
for the dual variable $\vct{z}_t$, the random matrices $(\mtx{\Omega}, \mtx{\Psi})$,
and the sketch $(\mtx{Y}, \mtx{W})$.
Owing to the black-box assumption~\eqref{eqn:A-action},
the algorithm completes the singular vector computations in~\eqref{eqn:CGM-sketch-dir}
with $\Theta(d + m + n)$ working storage.
At no point during the iteration do we instantiate an $m \times n$ matrix!
Arithmetic costs are on the same order as the standard version of CGM.

\begin{algorithm}[tb]
  \caption{SketchyCGM for model problem~\eqref{eqn:model-prob}
    \label{alg:sketch-CGM}}
  \begin{algorithmic}[1]
    \Require{Data for~\eqref{eqn:model-prob}; suboptimality $\eps$; target rank $r$}
    \Ensure{Rank-$r$ approximate solution $\hat{\mtx{X}} = \mtx{U\Sigma V}^*$
    of \eqref{eqn:model-prob} in factored form}
\vspace{0.5pc}

    \Function{SketchyCGM}{}
    	\State	\textsc{Sketch.Init}$(m, n, r)$
			\Comment{Initialize \textsc{Sketch} object}
		\State	$\vct{z} \gets \vct{0}$
		\For{$t \gets 0, 1, 2, 3, \dots$}
			\State	$(\vct{u}, \vct{v}) \gets \texttt{MaxSingVec}( \mathcal{A}^*(\grad f( \vct{z} ) ) )$
				\Comment{Use Lanczos method} 
			\State	$\vct{h} \gets \mathcal{A}(-\alpha \vct{uv}^*)$
			\If{$\ip{ \vct{z} - \vct{h} }{ \grad f(\vct{z}) } \leq \eps$}
				\textbf{break for}
				\Comment{Stop when $\eps$-suboptimal}
			\EndIf
			\State	$\eta \gets 2/(t+2)$
			\State	$\vct{z} \gets (1-\eta) \vct{z} + \eta \vct{h}$
			\State	$\textsc{Sketch.CGMUpdate}(-\alpha \vct{u}, \vct{v}, \eta)$
				\Comment{Update primal sketch}
		\EndFor
		\State $(\mtx{U}, \mtx{\Sigma}, \mtx{V}) \gets \textsc{Sketch.Reconstruct}(\,)$
		\State \Return{$(\mtx{U}, \mtx{\Sigma}, \mtx{V})$}
    \EndFunction

\vspace{0.5pc}

\ifdefined \istwocolumn
\centerline{--------\quad
\textbf{Methods for} \textsc{Sketch} \textbf{object}
\quad--------}
\else
\centerline{---------------------\quad
\textbf{Methods for} \textsc{Sketch} \textbf{object}
\quad---------------------}
\fi

\vspace{0.5pc}

    \Function{Sketch.Init}{$m$, $n$, $r$} 
    		\Comment{Rank-$r$ approx.~of $m \times n$ matrix}
		\State	$k \gets 2 r + 1$ and $\ell \gets 4 r + 3$
		\State	$\mtx{\Omega} \gets \texttt{randn}(n, k)$ and
				$\mtx{\Psi} \gets \texttt{randn}(\ell, m)$
		\State	$\mtx{Y} \gets \texttt{zeros}(m, k)$ and
				$\mtx{W} \gets \texttt{zeros}(\ell, n)$
    \EndFunction

    \vspace{0.25pc}

    \Function{Sketch.CGMUpdate}{$\vct{u}$, $\vct{v}$, $\eta$}
    	\State $\mtx{Y} \gets (1 - \eta) \mtx{Y} + \eta \vct{u} (\vct{v}^* \mtx{\Omega})$
			\Comment{Average $\vct{uv}^*$ into range sketch}
		\State $\mtx{W} \gets (1 - \eta) \mtx{W} + \eta (\mtx{\Psi} \vct{u}) \vct{v}^*$
			\Comment{Average $\vct{uv}^*$ into co-range sketch}
    \EndFunction

    \vspace{0.25pc}

    \Function{Sketch.Reconstruct}{\,}
    	\State	$\mtx{Q} \gets \texttt{orth}( \mtx{Y} )$
			\Comment{Orthobasis for range of $\mtx{Y}$}
		\State	$\mtx{B} \gets (\mtx{\Psi} \mtx{Q}) \backslash \mtx{W}$
			\Comment{Solve family of least-squares problems}
		\State	$(\mtx{U}, \mtx{\Sigma}, \mtx{V}) \gets \texttt{svds}(\mtx{B}, r)$
			\Comment{Compute full SVD; truncate to rank $r$}
		\State \Return{$(\mtx{QU}, \mtx{\Sigma}, \mtx{V})$}
			\Comment{Consolidate left unitary factor}
	\EndFunction

	\vspace{0.25pc}

\end{algorithmic}
\end{algorithm}

\subsection{Convergence Results for SketchyCGM}

SketchyCGM is a provably correct method for computing a
low-rank approximation of a solution to~\eqref{eqn:model-prob}.

\begin{theorem} \label{thm:low-rank-equivalence}
Suppose that the iterates $\mtx{X}_t$ from
the CGM iteration~\eqref{eqn:CGM-primal-init}--\eqref{eqn:CGM-primal-update}
converge to a matrix $\mtx{X}_{\rm cgm}$. 
Let $\hat{\mtx{X}}_t$ be the rank-$r$ reconstruction of $\mtx{X}_t$ produced by SketchyCGM.
Then
$$
\lim_{t \to \infty}
\Expect \fnorm{ \smash{ \hat{\mtx{X}}_{t} - \mtx{X}_{\rm cgm} } }
	\leq 3\sqrt{2} \fnorm{ \mtx{X}_{\rm cgm} - [ \mtx{X}_{\rm cgm} ]_r }.
$$
In particular, if $\rank(\mtx{X}_{\rm cgm}) \leq r$, then
$$
\Expect \fnorm{ \smash{ \hat{\mtx{X}}_{t} - \mtx{X}_{\rm cgm} } }
	\to 0.
$$
\end{theorem}

SketchyCGM always works in the fundamental
case where each solution of~\eqref{eqn:model-prob}
has low rank.

\begin{theorem} \label{thm:sketch-CGM-unique}
Suppose that the solution set $S_{\star}$ of the problem~\eqref{eqn:model-prob}
contains only matrices with rank $r$ or less.  Then SketchyCGM attains
$\Expect \dist_{\rm F}(\hat{\mtx{X}}_t, S_{\star}) \to 0$.
\end{theorem}

Suppose that the optimal point of~\eqref{eqn:model-prob}
is stable in the sense that the value of the objective
function increases as we depart from optimality. 
Then the SketchyCGM estimates converge
at a controlled rate.

\begin{theorem} \label{thm:cgm-sketch-rate}
Fix $\kappa > 0$ and $\nu > 0$.
Suppose the (unique) solution $\mtx{X}_{\star}$ of~\eqref{eqn:model-prob}
has~$\rank(\mtx{X}_{\star}) \leq r$ and
\begin{equation} \label{eqn:model-stable}
f(\mathcal{A}\mtx{X}) - f(\mathcal{A}\mtx{X}_{\star})
	\geq \kappa \fnorm{ \mtx{X} - \mtx{X}_{\star} }^{\nu}
\end{equation}
for all feasible $\mtx{X}$. 
Then we have the error bound
$$
\Expect \fnorm{ \smash{\hat{\mtx{X}}_t - \mtx{X}_{\star}} }
	\leq 6 \left( \frac{2 C \kappa^{-1} }{t + 2} \right)^{1/\nu}
	\ \text{for $t = 0, 1, 2, \dots$}
$$
where $C$ is the curvature constant~\cite[Eqn.~(3)]{Jag13:Revisiting-Frank-Wolfe}
of the problem~\eqref{eqn:model-prob}. 
\end{theorem}

\noindent
\ifdefined \hassupplement {
See Appendix~\ref{app:theory} for the proofs of these results.}
\else {
See the supplement 
for the
proofs of these results.
} \fi

\subsection{SketchyCGM for PSD Optimization}
\label{sec:CGM-psd}
Next, we present a generalization of the convex phase
retrieval problem~\eqref{eqn:phase-convex}.
Consider the convex template
\begin{equation}  \label{eqn:model-prob-psd}
\underset{\mtx{X} \in \C^{n\times n}}{\minimize} \ f(\mathcal{A}\mtx{X})
\quad\subjto\quad \trace \mtx{X} \leq \alpha, \ \mtx{X} \psdge \mtx{0}.
\end{equation}
As before, $\mathcal{A} : \C^{n \times n} \to \C^d$ is a linear map,
and $f : \C^d \to \R$ is a differentiable convex function.

It is easy to adapt SketchyCGM to handle~\eqref{eqn:model-prob-psd}
instead of~\eqref{eqn:model-prob}.
To sketch the complex psd matrix variable,
we follow the approach described in~\cite[Sec.~7.3]{TYUC17:Randomized-Single-View}.
We also make small changes to the SketchyCGM iteration.
Replace the computation~\eqref{eqn:CGM-sketch-dir} with
$$
\begin{aligned}
(\lambda_t, \vct{u}_t) & = \texttt{MinEig}(\mathcal{A}^* (\grad f(\vct{z}_t))); \\
\vct{h}_t & = \begin{cases}
	\mathcal{A}(\alpha\vct{u}_t \vct{u}_t^*), & \lambda_t \leq 0 \\
	\vct{0}, & \text{otherwise}.
\end{cases}
\end{aligned}
$$
\texttt{MinEig} returns the minimum eigenvalue $\lambda_t$
and an associated eigenvector $\vct{u}_t$
of a conjugate symmetric matrix.
This variant behaves the same as SketchyCGM.

\section{COMPUTATIONAL EVIDENCE}
\label{sec:numerics}

In this section, we demonstrate that SketchyCGM is a practical
algorithm for convex low-rank matrix optimization.
We focus on phase retrieval problems because they provide a
dramatic illustration of the power of storage-optimal convex optimization.
\ifdefined \hassupplement {
Appendix~\ref{app:more-numerics} adduces additional examples,
including some matrix completion problems.
} \else {
The supplementary material adduces additional examples,
including some matrix completion problems.
} \fi

\subsection{Synthetic Phase Retrieval Problems}
\label{sec:synthetic-PR}

To begin, we show that our approach to solving
the convex phase retrieval problem~\eqref{eqn:phase-convex}
has better memory scaling than other convex optimization methods.

We compare five convex optimization algorithms:
the classic proximal gradient method (PGM)~\cite{Roc76:Monotone-Operators};
the Auslender--Teboulle (AT) accelerated method~\cite{AT06:Interior-Gradient};
the classic CGM algorithm~\cite{Jag13:Revisiting-Frank-Wolfe};
a storage-efficient CGM variant (ThinCGM)~\cite{YHC15:Scalable-Convex} based on low-rank SVD updating; and
the psd variant of SketchyCGM from \S\ref{sec:CGM-psd} with rank parameter $r = 1$.

All five methods solve~\eqref{eqn:model-prob-psd} reliably.
The proximal methods (PGM and AT) perform a full eigenvalue decomposition
of the iterate at each step, but they can be accelerated by adaptively
choosing the number of eigenvectors to compute.
The methods based on CGM only need the top eigenvector,
so they perform less arithmetic per iteration.

To compare the storage costs of the five algorithms,
let us consider a synthetic phase retrieval problem.
We draw a vector $\vct{x}_{\natural} \in \C^n$ from
the complex standard normal distribution.
Then we acquire $d = 10 n$ phaseless measurements~\eqref{eqn:phase-retrieval-meas},
corrupted with independent Gaussian noise so that the SNR is 20 dB.
The measurement vectors $\vct{a}_i$ derive from a coded
diffraction pattern; 
\ifdefined \hassupplement {
see \S\ref{sec:synthetic-app} for details.}
\else {
see the supplement 
for details.} \fi
We solve the convex problem~\eqref{eqn:phase-convex}
with $\alpha$ equal to the average of the measurements $b_i$;
see~\cite[Sec.~II]{YHC15:Scalable-Convex}.
Then we compute a top eigenvector $\vct{x}_{\star}$ of the solution.  

Figure~\ref{fig:SpaceConv}(\textsc{a}) displays storage costs for each algorithm
as the signal length $n$ increases.
\ifdefined \hassupplement {
See \S\ref{sec:synthetic-app} for the numerical data.} \fi
We approximate memory usage by
reporting the total workspace allocated by MATLAB for the algorithm.
PGM, AT, and CGM have static allocations,
but 
they use a matrix variable of size $n^2$.
ThinCGM attempts to maintain a low-rank
approximation of the decision variable, but the rank increases
steadily, so the algorithm fails after $n = 10^5$.
In contrast, SketchyCGM has a static memory allocation
of $\Theta(n)$.  It already offers the best memory
footprint for $n = 10^2$, and it still works when $n = 10^6$.

In fact, SketchyCGM can tackle even larger problems.
We were able to reconstruct an image
with 
$n = 3,\!264 \times 2,\!448 \approx 7.99 \cdot 10^6$ pixels,
treated as a vector $\vct{x}_{\natural} \in \C^n$, 
given $d = 20n$ noiseless coded diffraction measurements, as above.
Figure~\ref{fig:SpaceConv}(\textsc{b}) plots the convergence
of the relative error:
$
\min_{\phi \in \R} \norm{\smash{ \econst^{\mathrm{i} \phi} \hat{\vct{x}}_t - \vct{x}_{\natural}}} / \norm{\vct{x}_{\natural}}$,
where $\hat{\vct{x}}_t$ is a top eigenvector of the SketchyCGM iterate $\hat{\mtx{X}}_t$.
After 150 iterations,
the algorithm produced an image 
with relative error 
$0.0290$ and with PSNR 36.19 dB.
Thus, we can solve~\eqref{eqn:model-prob-psd}
when the psd matrix variable $\mtx{X}$ has $6.38 \cdot 10^{13}$ complex entries!

As compared to other convex optimization algorithms,
the main weakness of SketchyCGM is the $\mathcal{O}(\eps^{-1})$
iteration count.  Some algorithms, such as AT, can achieve $\mathcal{O}(\eps^{-1/2})$
iteration count, but they are limited to smaller problems.
Closing this gap is an important question for future work.

\begin{figure*}[t!]
    \centering
    \begin{subfigure}{0.6\textwidth}
        \centering
        \includegraphics[height=4.75cm]{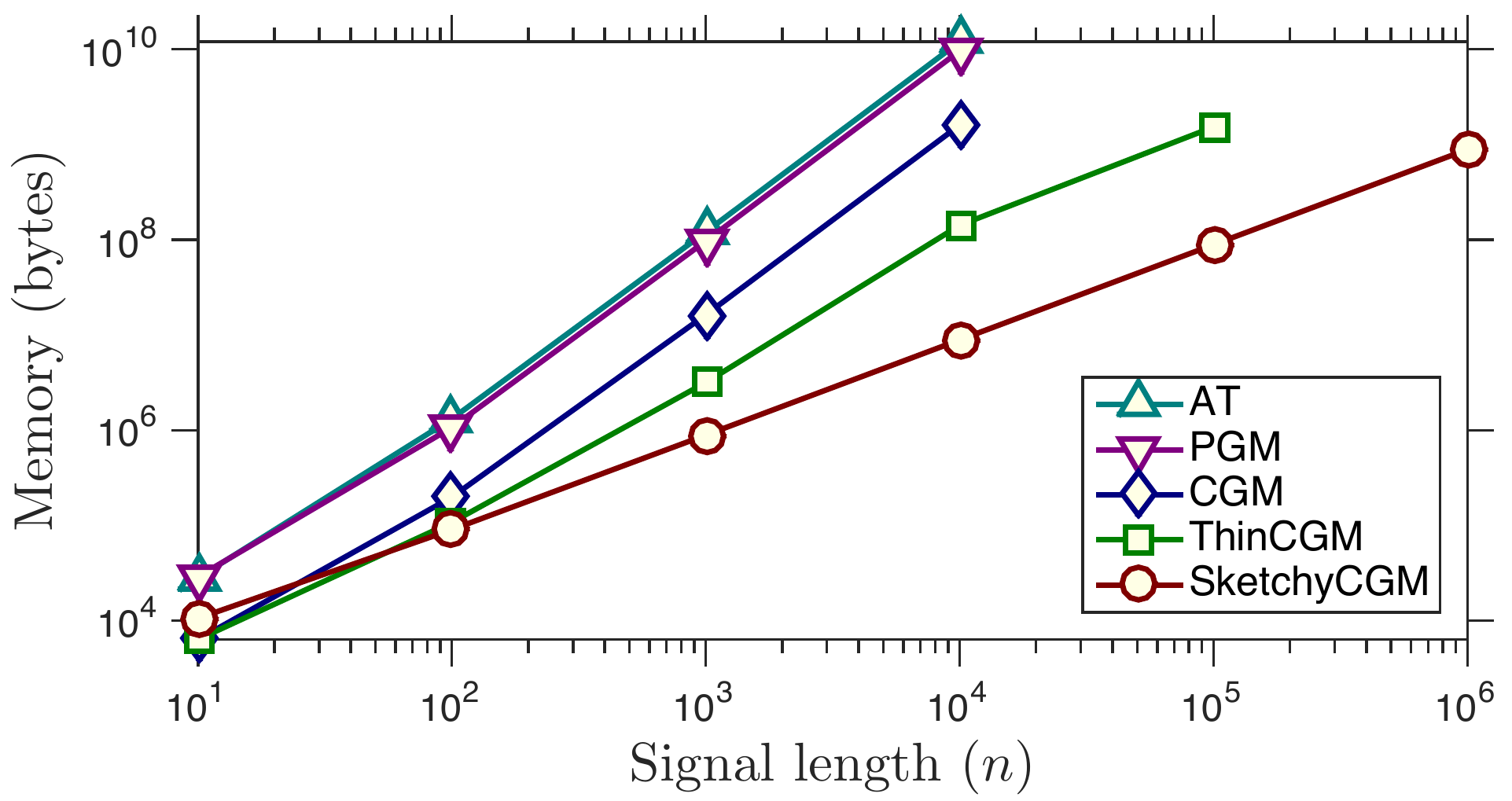}
        \caption{Memory usage for five algorithms}
    \end{subfigure}
    \hfill
    \begin{subfigure}{0.35\textwidth}
        \centering
        \includegraphics[height=4.75cm]{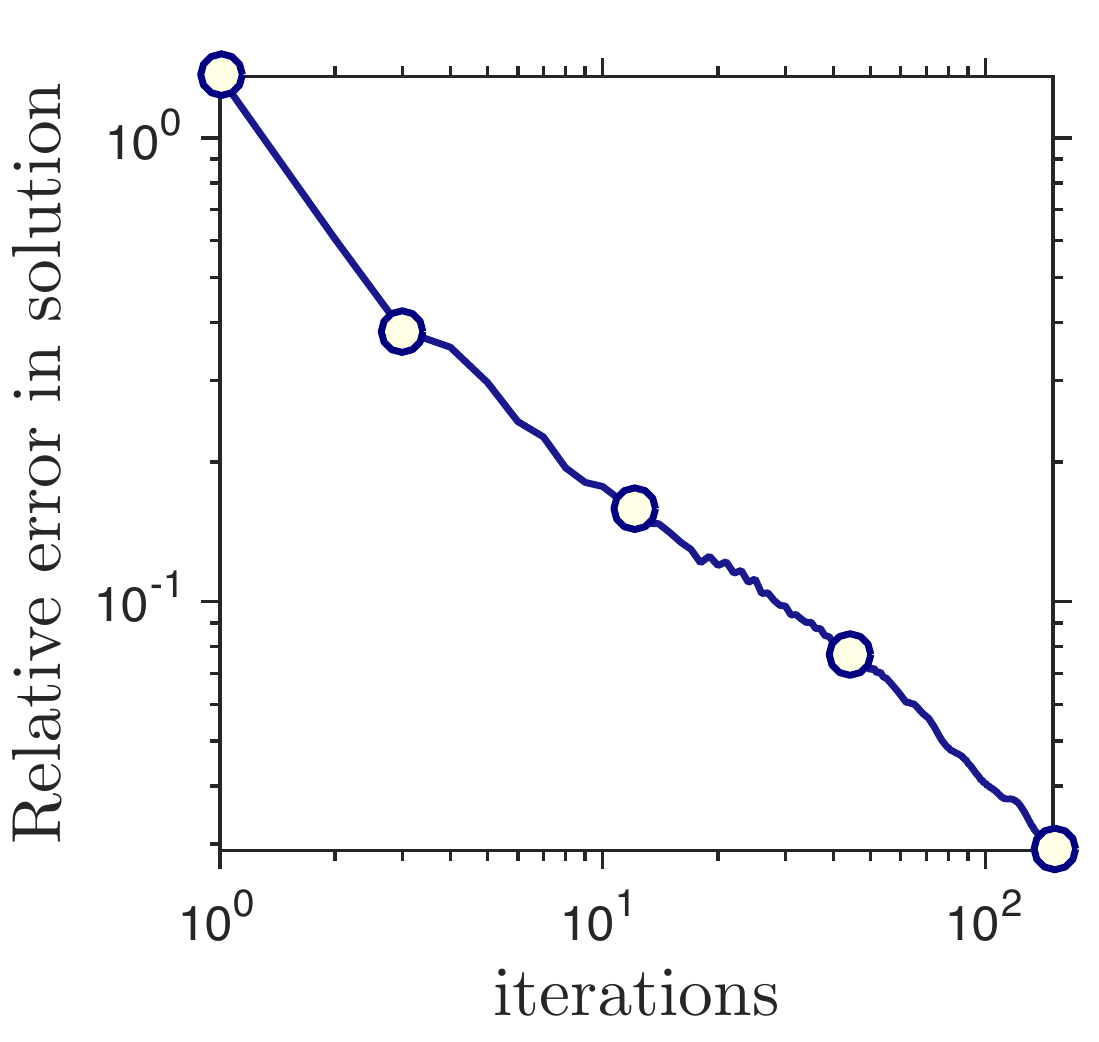}
        \caption{Convergence on largest problem}
    \end{subfigure}
\caption{\textsl{Memory Usage and Convergence.}  \textsc{(a)}
Memory scaling for five convex optimization algorithms applied
to a synthetic instance of the convex phase retrieval problem~\eqref{eqn:phase-convex}.
\textsc{(b)} Relative $\ell_2$ error achieved by SketchyCGM for convex phase retrieval
of a synthetic signal of length $n = 8 \cdot 10^6$.
See \S\ref{sec:fourier-ptychography-problem} for further details.}
\vspace{1pc}
\label{fig:SpaceConv} 
\end{figure*}

\begin{figure*}[t!]
    \centering
    \begin{subfigure}{0.3\textwidth}
        \centering
	\begin{tikzpicture}
    		\node[anchor=south west,inner sep=0] (image) at (0,0) {\includegraphics[width=\textwidth]{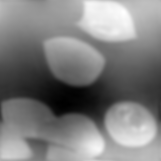}};
    		\begin{scope}[x={(image.south east)},y={(image.north west)}]
        			\draw[red!70] (0.50,0.75) rectangle (0.80,0.90);
        			\draw[red!70] (0.70,0.14) rectangle (0.95,0.38);
    		\end{scope}
	\end{tikzpicture}
        \caption{SketchyCGM}
    \end{subfigure}
    ~
    \begin{subfigure}{0.3\textwidth}
        \centering
        \includegraphics[width=\textwidth]{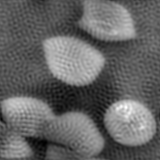}
        \caption{Burer--Monteiro~\cite{BM03:Nonlinear-Programming,HCO+15:Solving-Ptychography}}
    \end{subfigure}
    ~
    \begin{subfigure}{0.3\textwidth}
        \centering
        \includegraphics[width=\textwidth]{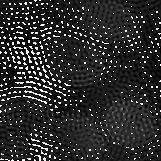}
        \caption{Wirtinger Flow~\cite{CLS15:Phase-Retrieval}}
    \end{subfigure}
\caption{\textsl{Imaging by Fourier Ptychography.}
 Three algorithms for Fourier ptychographic imaging
 via phase retrieval.  Brightness indicates the
 complex phase of a pixel, which roughly corresponds
 with the thickness of the sample.  Only relative
 differences in brightness are meaningful.
 Red boxes mark malaria parasites in blood cells.}
 \label{fig:ptych}
 \end{figure*}

\subsection{Fourier Ptychography}
\label{sec:fourier-ptychography-problem}

Up to now, it has not been possible to attack phase retrieval problems
of a realistic size by solving the convex formulation~\eqref{eqn:phase-convex}.
As we have shown, current convex optimization algorithms cannot achieve scale.
Instead, many researchers apply nonconvex heuristics to solve phase retrieval
problems~\cite{Fie82:Phase-Retrieval,
NJS15:Phase-Retrieval,CLS15:Phase-Retrieval,HCO+15:Solving-Ptychography}. 
These heuristics can produce reasonable solutions,
but they require extensive tuning and have limited effectiveness.
In this section, we demonstrate that, without any modification, SketchyCGM
can solve a phase retrieval problem from an application in microscopy.
Furthermore, it produces an image that is superior
to the results obtained using major nonconvex heuristics.

We study a phase retrieval problem that arises
from an imaging modality called Fourier ptychography (FP)~\cite{HCO+15:Solving-Ptychography}.
The authors of~\cite{HCO+15:Solving-Ptychography} provided measurements
of a slide containing human blood cells from a working FP system.
We treat the sample as an image with $n = 25,\!600$ pixels, which
we represent as a vector $\vct{x}_{\natural} \in \C^n$.
The goal is to reconstruct the \emph{phase} of the image,
which roughly corresponds with the thickness of the sample
at a given location.

The data consists of 29 illuminations, each containing $6,\!400$ pixels.
The number of measurements $d = 185,\!600$.
The measurement vectors $\vct{a}_i$ are obtained
from windowed discrete Fourier transforms.
We can formulate the problem of reconstructing 
the sample $\vct{x}_{\natural}$ using the convex phase retrieval template~\eqref{eqn:phase-convex}
with the parameter $\alpha = 1,\!400$.  In this instance,
the psd matrix variable $\mtx{X} \in \C^{n \times n}$ 
has $6.55 \cdot 10^8$ complex entries.

To solve~\eqref{eqn:phase-convex},
we run the SketchyCGM variant from \S\ref{sec:CGM-psd}
with the rank parameter $r = 1$ for $10,\!000$ iterations. 
We factor the rank-one matrix output to obtain an approximation $\vct{x}_{\star}$
of the sample.  Figure~\ref{fig:ptych}(\textsc{a}) displays the
phase of the reconstruction $\vct{x}_{\star}$.

Figure~\ref{fig:ptych} also includes comparisons with two nonconvex heuristics.
The authors of~\cite{HCO+15:Solving-Ptychography} provided a reconstruction
obtained via the Burer--Monteiro method~\cite{BM03:Nonlinear-Programming}.
We also applied Wirtinger Flow~\cite{CLS15:Phase-Retrieval} with the recommended parameters.
SketchyCGM yields a smooth and detailed phase reconstruction.
Burer--Monteiro produces severe artifacts, which suggest an unphysical
oscillation in the thickness of the sample.  Wirtinger Flow fails completely.
These results are consistent with~\cite[Fig.~4]{HCO+15:Solving-Ptychography},
which indicates 5--10 dB improvement of convex optimization over heuristics.

The quality of phase reconstruction can be essential for scientific purposes.
In this particular example, some of the blood cells are infected with
malaria parasites (Figure~\ref{fig:ptych}(\textsc{a}), red boxes).  Diagnosis is
easier when the visual acuity of the reconstruction is high.

\ifdefined \hassupplement {
Appendix~\ref{app:PR} contains additional details and results on Fourier ptychographic imaging.} 
\else {
The supplement 
contains further details and results on Fourier ptychographic imaging.} \fi

\section{DISCUSSION}

We have shown that it is possible to construct a low-rank
approximate solution to a large-scale matrix optimization problem
by sketching the decision variable.
Let us contrast our approach against other low-storage techniques for large-scale optimization.

\myparagraph{Sketchy Decisions.}
To the best of our knowledge, there are no direct precedents
for the idea and realization of an optimization algorithm that sketches the decision variable.
This work does partake in a broader vision that randomization can be used to design numerical
algorithms~\citep{HMT11:Finding-Structure,Mah11:Randomized-Algorithms,Woo14:Sketching-Tool}.

Researchers have considered \textbf{sketching the problem data}
as a way to reduce the size of a problem specification in exchange for additional error.
This idea dates to the paper of Sarl{\'o}s~\citep{Sar06:Improved-Approximation};
see also~\citep{Mah11:Randomized-Algorithms,Woo14:Sketching-Tool,PW15:Randomized-Sketches}.
There are also several papers~\cite{PW15:Newton-Sketch,EM15:Convergence-Rates,ABH16:Second-Order-Stochastic}
in which researchers try to improve the computational or storage footprint of convex optimization methods by \textbf{sketching internal variables}, such as Hessians.

None of these approaches address the core issue that concerns us:
the decision variable may require much more storage than the solution.

\myparagraph{Dropping Nonconvexity.}
We have already discussed a major trend in which researchers develop algorithms
that attack \textbf{nonconvex reformulations} of a problem.
For example, see~\cite{BM03:Nonlinear-Programming,JNS13:Low-Rank-Matrix, 
BKS16:Dropping-Convexity, BVB16:Nonconvex-Burer-Monteiro}.
The main advantage is to reduce the size of the decision variable;
some methods also have the ancillary benefit of rapid local convergence.
On the other hand, these algorithms are provably correct only under
strong statistical assumptions on the problem data.

\myparagraph{Prospects.}
Our work shows that convex optimization need not have
high storage complexity for problems with a compact specification
and a simple solution.  In particular, for low-rank matrix optimization,
storage is no reason to drop convexity.

It has not escaped our notice that the specific pairing of sketching and CGM
that we have postulated immediately suggests a possible mechanism for solving
other structured convex programs using optimal storage.

\subsubsection*{Acknowledgments}

JAT and MU were supported in part by ONR Awards N00014-11-1-0025 and N00014-17-1-2146
and the Gordon \& Betty Moore Foundation.
VC and AY were supported in part by the European Commission under
Grant ERC Future Proof, SNF 200021-146750, and SNF CRSII2-147633.
The authors thank Dr. Roarke Horstmeyer for sharing the blood cell data. 

\clearpage
\appendix

\section{Convergence of SketchyCGM}
\label{app:theory}

In this appendix, we develop a basic convergence theory
for the SketchyCGM algorithm.  We focus on situations
where the low-rank estimates produced by SketchyCGM converge
to a low-rank solution of the model problem~\eqref{eqn:model-prob}.

\paragraph{Preliminaries.}
We rely on the following standard convergence result for CGM.

\begin{fact}[CGM: Convergence rate for objective] \label{fact:CGM-objective}
Let $\mtx{X}_{\star}$ be an arbitrary solution to~\eqref{eqn:model-prob}.
For each $t \geq 0$, the matrix $\mtx{X}_t$ given
by the CGM iteration~\eqref{eqn:CGM-primal-init}--\eqref{eqn:CGM-primal-update} satisfies
$$
f(\mathcal{A}\mtx{X}_t) - f(\mathcal{A}\mtx{X}_{\star})
	\leq \frac{2 C}{2 + t}.
$$
The number $C$
is called the \emph{curvature constant}~\citep[Eqn.~(3)]{Jag13:Revisiting-Frank-Wolfe}
of the problem~\eqref{eqn:model-prob}.
\end{fact}

\paragraph{Theorem~\ref{thm:low-rank-equivalence}: A basic convergence result.}
First, we study the case where the standard CGM
iteration~\eqref{eqn:CGM-primal-init}--\eqref{eqn:CGM-primal-update}
converges.  In this setting, we can show that SketchyCGM produces
iterates that tend toward a matrix close to the limiting value of CGM.

\begin{proof}[Proof of Theorem~\ref{thm:low-rank-equivalence}] 
According to the triangle inequality,
$$
\Expect \fnorm{ \smash{ \hat{\mtx{X}}_t - \mtx{X}_{\rm cgm} } }
	\leq \Expect \fnorm{ \smash{ \hat{\mtx{X}}_t - \mtx{X}_t } }
	+ \fnorm{ \mtx{X}_t - \mtx{X}_{\rm cgm} }.
$$
The error bound for reconstruction from the sketch, Theorem~\ref{thm:sketch-err-body},
shows that
$$
\begin{aligned}
\Expect \fnorm{ \smash{ \hat{\mtx{X}}_t - \mtx{X}_t } }
	& \leq 3\sqrt{2} \fnorm{ \mtx{X}_t - [\mtx{X}_t]_r } \\
	& \leq 3\sqrt{2} \fnorm{ \mtx{X}_t - [\mtx{X}_{\rm cgm}]_r }.
\end{aligned}
$$
The second inequality holds because $[\mtx{X}_t]$ is a best rank-$r$
approximation of $\mtx{X}_t$ with respect to Frobenius norm, whereas
$[\mtx{X}_{\rm cgm}]_r$ is an undistinguished rank-$r$ matrix.

Combine the last two displays to obtain
$$
\begin{aligned}
\lim_{t\to\infty} & \Expect \fnorm{ \smash{ \hat{\mtx{X}}_t - \mtx{X}_{\rm cgm} } } \\
	&\leq \lim_{t\to\infty} \big( \Expect \fnorm{ \smash{ \hat{\mtx{X}}_t - \mtx{X}_t } } 
	+ \fnorm{ \mtx{X}_t - \mtx{X}_{\rm cgm} } \big) \\
	& \leq \lim_{t\to\infty} 3\sqrt{2} \fnorm{ \mtx{X}_t - [\mtx{X}_{\rm cgm}]_r } \\
	& = 3\sqrt{2} \fnorm{ \mtx{X}_{\rm cgm} - [\mtx{X}_{\rm cgm}]_r }.
\end{aligned}
$$
The second inequality and the last line follow from the limit $\mtx{X}_{t} \to \mtx{X}_{\mathrm{cgm}}$.

If $\rank(\mtx{X}_{\rm cgm}) \leq r$, then 
$\mtx{X}_{\rm cgm} = [\mtx{X}_{\rm cgm}]_r$.  Therefore, we may conclude that
$\Expect \fnorm{ \smash{ \hat{\mtx{X}}_t - \mtx{X}_{\rm cgm} } } \to 0$.
\end{proof}

\paragraph{Theorem~\ref{thm:sketch-CGM-unique}: When all solutions are low rank.}
Next, we examine the situation where all of the solutions
to the model problem~\eqref{eqn:model-prob} have low rank.
In this case, we can show that SketchyCGM produces a
sequence of approximations that approaches the solution
set of the problem.  This point does not follow formally
from Theorem~\ref{thm:low-rank-equivalence} because
CGM need not produce a convergent sequence of iterates.

\begin{proof}[Proof of Theorem~\ref{thm:sketch-CGM-unique}]
As a consequence of the triangle inequality,
$$
\Expect \dist_{\rm F}(\hat{\mtx{X}}_t, S_{\star})
	\leq \Expect \fnorm{\smash{\hat{\mtx{X}}_t - \mtx{X}_t}}
	+ \dist_{\rm F}( \mtx{X}_t, S_{\star} ).
$$
We claim that the second term $\dist_{\rm F}(\mtx{X}_t, S_{\star}) \to 0$.
If so, then the first term converges to zero:
\begin{align*}
\Expect \fnorm{\smash{\hat{\mtx{X}}_t - \mtx{X}_t}}
	& \leq 3\sqrt{2} \fnorm{ \mtx{X}_t - [\mtx{X}_t]_r } \\
	& \leq 3\sqrt{2} \dist_{\rm F}(\mtx{X}_t, S_{\star})
	\to 0.
\end{align*}
The first inequality is Theorem~\ref{thm:sketch-err-body}.
The second bound holds because $[\mtx{X}_t]_r$ is
a best rank-$r$ approximation of $\mtx{X}_t$, while
$S_{\star}$ is an unremarkable set of rank-$r$ matrices.
These observations complete the proof.

Let us turn to the claim.
Abbreviate the objective function of~\eqref{eqn:model-prob}
as $g = f \circ \mathcal{A}$.  The continuous function $g$
attains its minimal value $g_{\star}$ on the compact feasible
set of~\eqref{eqn:model-prob}.  The standard convergence result for CGM,
Fact~\ref{fact:CGM-objective}, shows that $g(\mtx{X}_t) \to g_{\star}$.
Now, fix a number $\delta > 0$.  Define 
\begin{gather*}
E = \big\{ \mtx{X} \in \R^{m \times n} : \norm{\mtx{X}}_{S_1} \leq \alpha;\
	\dist(\mtx{X}, S_{\star}) \geq \delta \big\}; \\
v = \inf\{ g(\mtx{X}) : \mtx{X} \in E \}.
\end{gather*}
If $E$ is empty, then $v = + \infty$.  Otherwise,
the continuous function $g$ attains the value $v$ on the compact set $E$.
In either case, $v > g_{\star}$ because $E$ contains no optimal point of~\eqref{eqn:model-prob}.
Since $g(\mtx{X}_t) \to g_{\star}$, we must have $g(\mtx{X}_t) < v$
for all sufficiently large $t$.  Therefore, $\mtx{X}_t \notin E$
for large $t$.  We conclude that
$\dist_{\rm F}(\mtx{X}_t, S_{\star}) < \delta$ for large $t$,
as required.
\end{proof}

\paragraph{Theorem~\ref{thm:cgm-sketch-rate}: Rate of convergence.}
Finally, we identify a setting where we can bound the rate
at which the estimates produced by SketchyCGM converge to
an optimal point of~\eqref{eqn:model-prob}.  To do so, we
assume that the optimal point is stable in the sense that
feasible perturbations away from optimality are reflected
in increases in the value of the objective function.

\begin{proof}[Proof of Theorem~\ref{thm:cgm-sketch-rate}]
Since the CGM iterate $\mtx{X}_t$ is feasible for~\eqref{eqn:model-prob},
we can use the stability hypothesis~\eqref{eqn:model-stable} to calculate that
$$
\begin{aligned}
f(\mathcal{A}\mtx{X}_t) - f( \mathcal{A} \mtx{X}_{\star} )
	& \geq \kappa \fnorm{ \mtx{X}_t - \mtx{X}_{\star} }^{\nu} \\
	& \geq \kappa \fnorm{ \mtx{X}_t - [\mtx{X}_t]_r }^{\nu} \\
	& \geq  \kappa \big[ (3\sqrt{2})^{-1} \Expect \fnorm{ \smash{\mtx{X}_t - \hat{\mtx{X}}_t} } \big]^{\nu}.
\end{aligned}
$$
The second inequality holds because $\mtx{X}_{\star}$ has rank $r$,
while $[\mtx{X}_t]_r$ is a better rank-$r$ approximation of $\mtx{X}_t$.
The last inequality follows from Theorem~\ref{thm:sketch-err-body}.

The latter display implies that
\begin{multline*}
\Expect \norm{ \smash{\hat{\mtx{X}}_t - \mtx{X}_{\star}} }
	\leq \Expect \norm{ \smash{\hat{\mtx{X}}_t - \mtx{X}_t} } + \norm{ \mtx{X}_t - \mtx{X}_{\star} } \\
	\leq {} (3\sqrt{2} + 1)\big[ \kappa^{-1} \big( f(\mathcal{A}\mtx{X}_t) - f(\mathcal{A}\mtx{X}_{\star}) \big) \big]^{1/\nu}.
\end{multline*}
To complete the proof, invoke the standard convergence
result 
for CGM, Fact~\ref{fact:CGM-objective}, to bound the quantity
$f(\mathcal{A}\mtx{X}_t) - f(\mathcal{A}\mtx{X}_{\star})$.
Last, simplify the constant.
\end{proof}

\section{More computational evidence}
\label{app:more-numerics}

This appendix elaborates on the phase retrieval experiments
described in \S\ref{sec:numerics}.  It also presents
additional experiments on matrix completion and phase retrieval.

\subsection{Loss functions}

Our experiments involve a number of different models for data,
so we require several elementary loss functions.  Each of these
maps is an extended convex function $\psi : \R \times \R \to \R$.
Define
$$
\begin{aligned}
\psi_{\mathrm{gauss}}(z; b) \ &=\ \tfrac{1}{2} (z - b)^2; \\
\psi_{\mathrm{huber}}(z; b) \ &=\ \huber(z, b); \\
\psi_{\mathrm{logistic}}(z; b) \ &= \ \log(1 + \exp(-bz)); \\
\psi_{\mathrm{poisson}}(z; b) \ &= \ z - b \log z. 
\end{aligned}
$$
The objectives correspond, respectively, to the negative
log-likelihood of observing the data $b$ under a Gaussian
model, a Gauss--Laplace model, a Bernoulli model, and a
Poisson model.

\subsection{Matrix completion}
\label{sec:mc-app}

One principal advantage of SketchyCGM is its flexibility.
Without any modification, the algorithm can provably solve
any convex optimization problem with a smooth objective
and a Schatten 1-norm constraint.  To demonstrate this
point numerically, we present the results of fitting the
benchmark MovieLens 100K and 10M datasets~\cite{HK15:MovieLens-Datasets}
with three different loss functions.

The MovieLens $N$ dataset contains (about) $N$ ratings that
users of a website assigned to a collection of movies.
Suppose that there are $m$ users and $n$ movies.
The data consists of triples $(i, j, b)$ where
$i \in \{1, \dots, m\}$ is a user,
$j \in \{1, \dots, n\}$ is a movie,
and $b \in \{1, 2, 3, 4, 5\}$ is the rating
of movie $j$ by user $i$.

We preprocess the data minimally.  We remove the movies that are not rated by
any user, as well as the users that have not provided any ratings.
To fit a logistic model, we also binarize the ratings by replacing all values
above 3.5 with $1$ and the rest with $-1$.  Thus, the logistic model
seeks a classifier that separates high ratings (4 and 5) from low ratings
(1, 2, and 3).

We can use low-rank matrix completion to fit a model to the MovieLens data.
To see why, introduce a target matrix $\mtx{X}_{\natural} \in \R^{m \times n}$
that tabulates all ratings (known and unknown) of movies by users.
One may imagine that $\mtx{X}_{\natural}$ has low rank because a lot of the variation in
the ratings is explained by the quality of the movie, its genre, and a user's
preference for that genre.  We only observe a subset of the entries of $\mtx{X}_{\natural}$,
and our goal is to predict the remaining entries.

We model this matrix completion problem using the formulation~\eqref{eqn:model-prob}.
Let $E$ be a training set of user--movie pairs,
and let $\vct{b} \in \R^E$ list the associated ratings.
Introduce the linear map $\mathcal{A} : \R^{m \times n} \to \R^E$
where
$$
\mathcal{A} \mtx{X} = \{ x_{ij} : (i, j) \in E \}.
$$
The objective function $f : \R^E \to \R$ has the form
$$
f_{\ell}(\vct{z}) \ = \ \frac{1}{\abs{E}}\sum_{(i, j) \in E} \psi_{\ell}(z_{ij}, b_{ij}),
$$
where $\ell \in \{ \mathrm{gauss}, \mathrm{huber}, \mathrm{logistic} \}$.

Suppose that $\mtx{X}_{\star} \in \R^{m \times n}$ is an estimate
for the target matrix $\mtx{X}_{\natural}$.  Let $E'$ be the test set of
user--movie pairs, with ratings listed in $\vct{b}' \in \R^{E'}$.  Define the
test error
$$
\mathrm{test}_{\ell}(\mtx{X}_{\star}) \ = \ 
	\frac{1}{\abs{E'}} \sum_{(i, j) \in E'} \psi_{\ell}( (\mtx{X}_{\star})_{ij}, b'_{ij} ).
$$
Once again, $\ell \in \{ \mathrm{gauss}, \mathrm{huber}, \mathrm{logistic} \}$.

For the 100K dataset, we use the default \texttt{ub} partition of the data
into training and test sets.  For each loss function, we sweep
$\alpha$ from $3,\!000$ to $10,\!000$ in steps of $500$.
As expected, the rank of a solution $\mtx{X}_{\star}$ of the problem~\eqref{eqn:model-prob}
increases with $\alpha$.  We select the value of $\alpha$ that provides
the best test error after $10,\!000$ iterations of CGM.

A similar procedure applies to the 10M dataset, with the default \texttt{rb} partition.
This time, we sweep $\alpha$ from $50,\!000$ to $250,\!000$ in steps of $25,\!000$.

For each dataset and each loss function, we fit the model \eqref{eqn:model-prob}
with the designated value of $\alpha$ by applying SketchyCGM.
Figures~\ref{f-movielens} and \ref{f-movielens-10m} show how the test error
for the SketchyCGM reconstruction varies as we change the rank parameter $r$ in SketchyCGM.
For the 100K dataset, rank $r = 50$ yields test error similar with
the CGM solution.  For the 10M dataset, rank $r = 200$ yields equivalent performance.

\begin{figure*}[h]
\begin{subfigure}{0.32\textwidth}
\includegraphics[width = \textwidth]{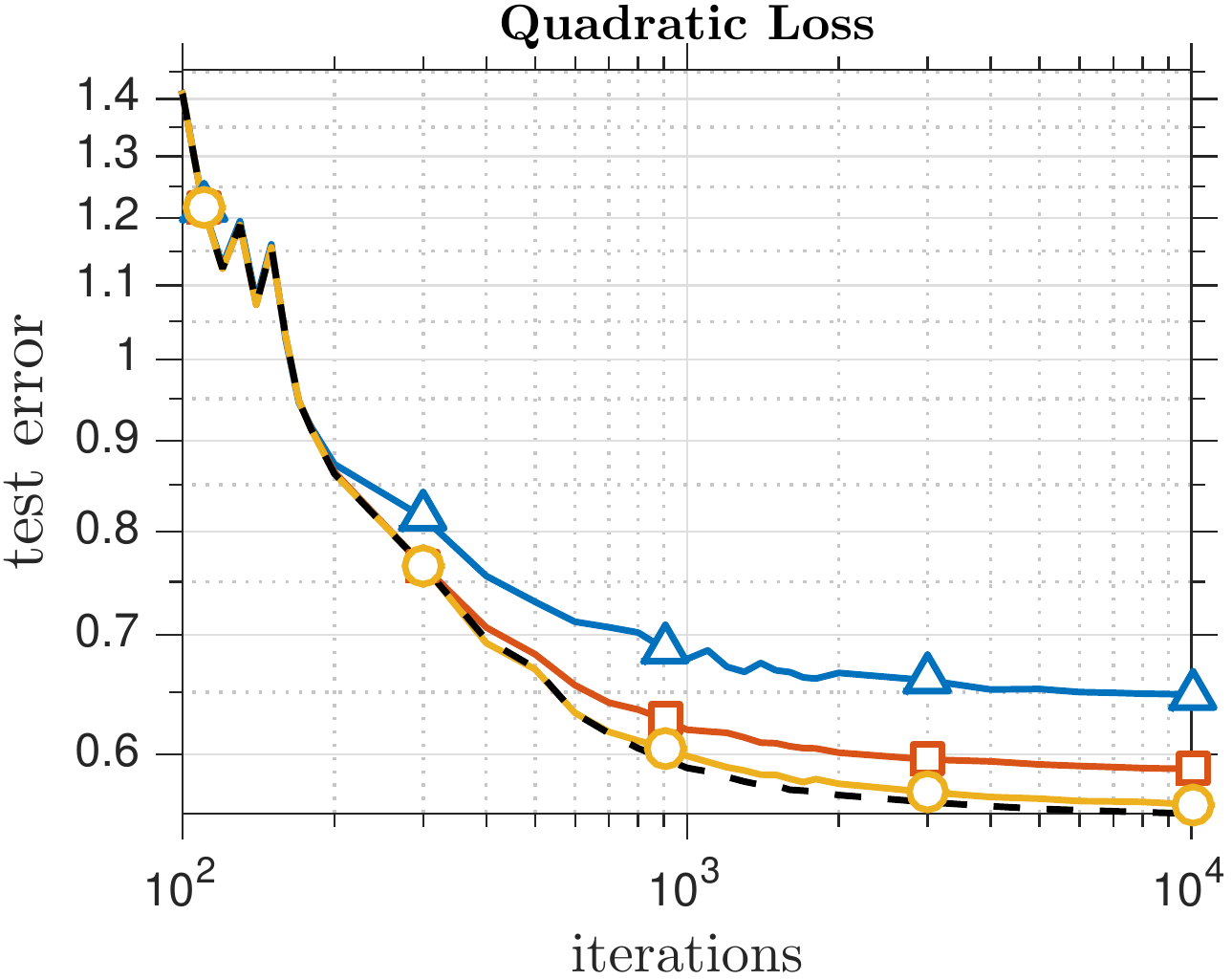}
\end{subfigure}
~
\begin{subfigure}{0.32\textwidth}
\includegraphics[width = \textwidth]{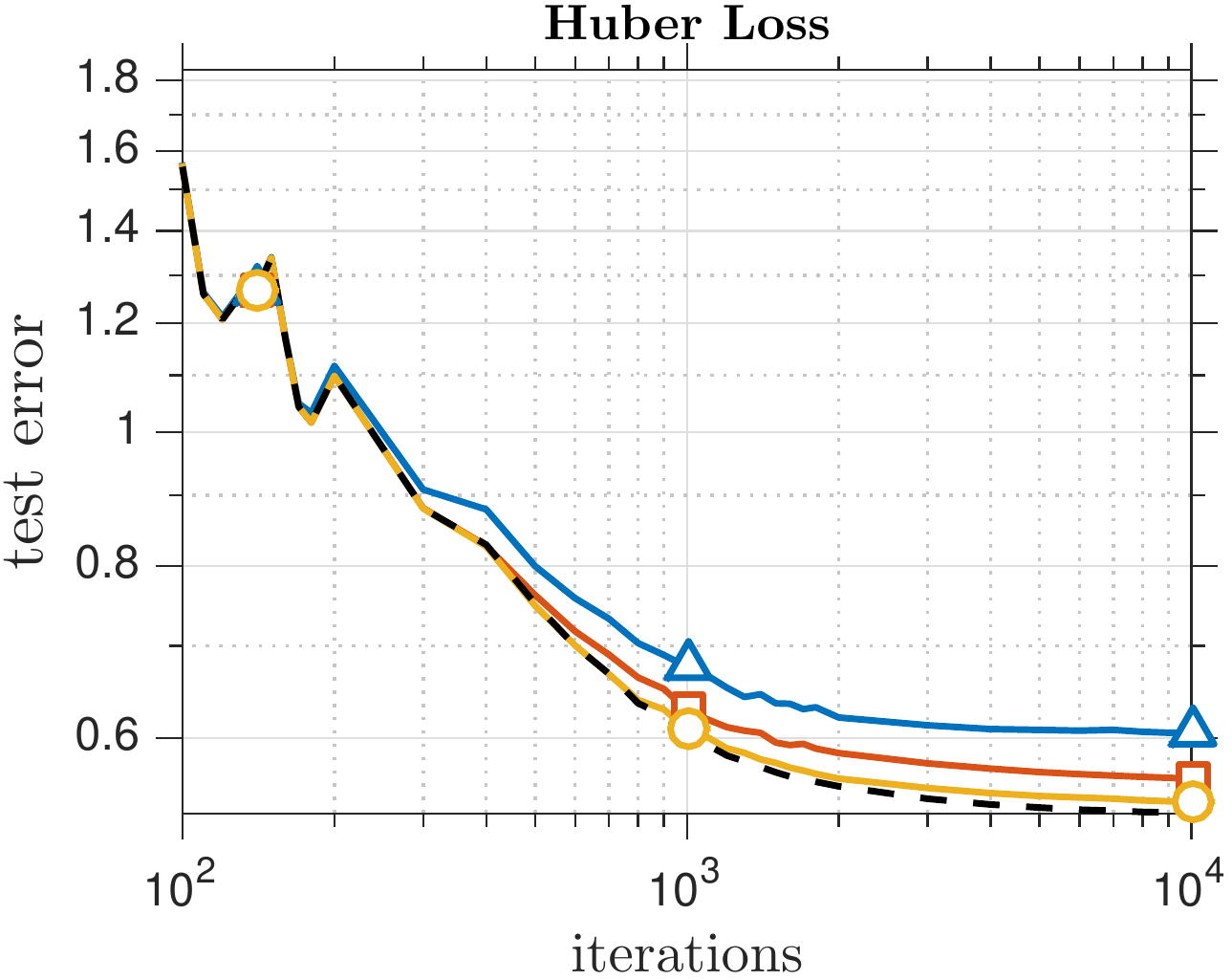}
\end{subfigure}
~
\begin{subfigure}{0.32\textwidth}
\includegraphics[width = \textwidth]{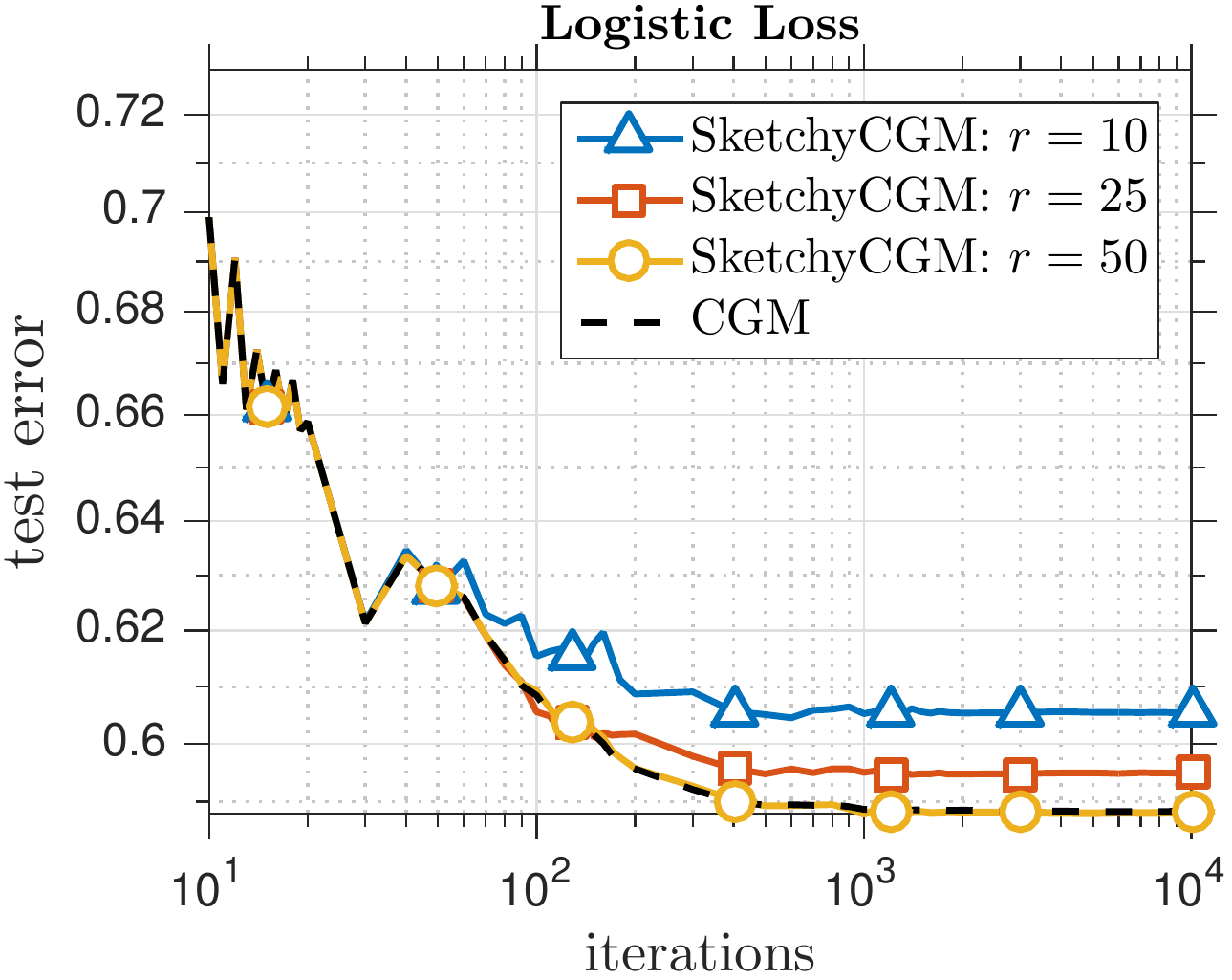}
\end{subfigure}
\caption{\label{f-movielens}\textsl{MovieLens 100K Dataset.}
Convergence of test error for CGM and SketchyCGM solutions to \eqref{eqn:model-prob} with three loss functions.
The parameter $\alpha$ is chosen by cross-validation on the final CGM solution: $\alpha = 7,\!000$ for quadratic loss, $7,\!500$ for Huber loss, and $4,\!500$ for logistic loss.  See \S\ref{sec:mc-app} for details.}
\end{figure*}

\begin{figure*}[h]
\begin{subfigure}{0.32\textwidth}
\includegraphics[width = \textwidth]{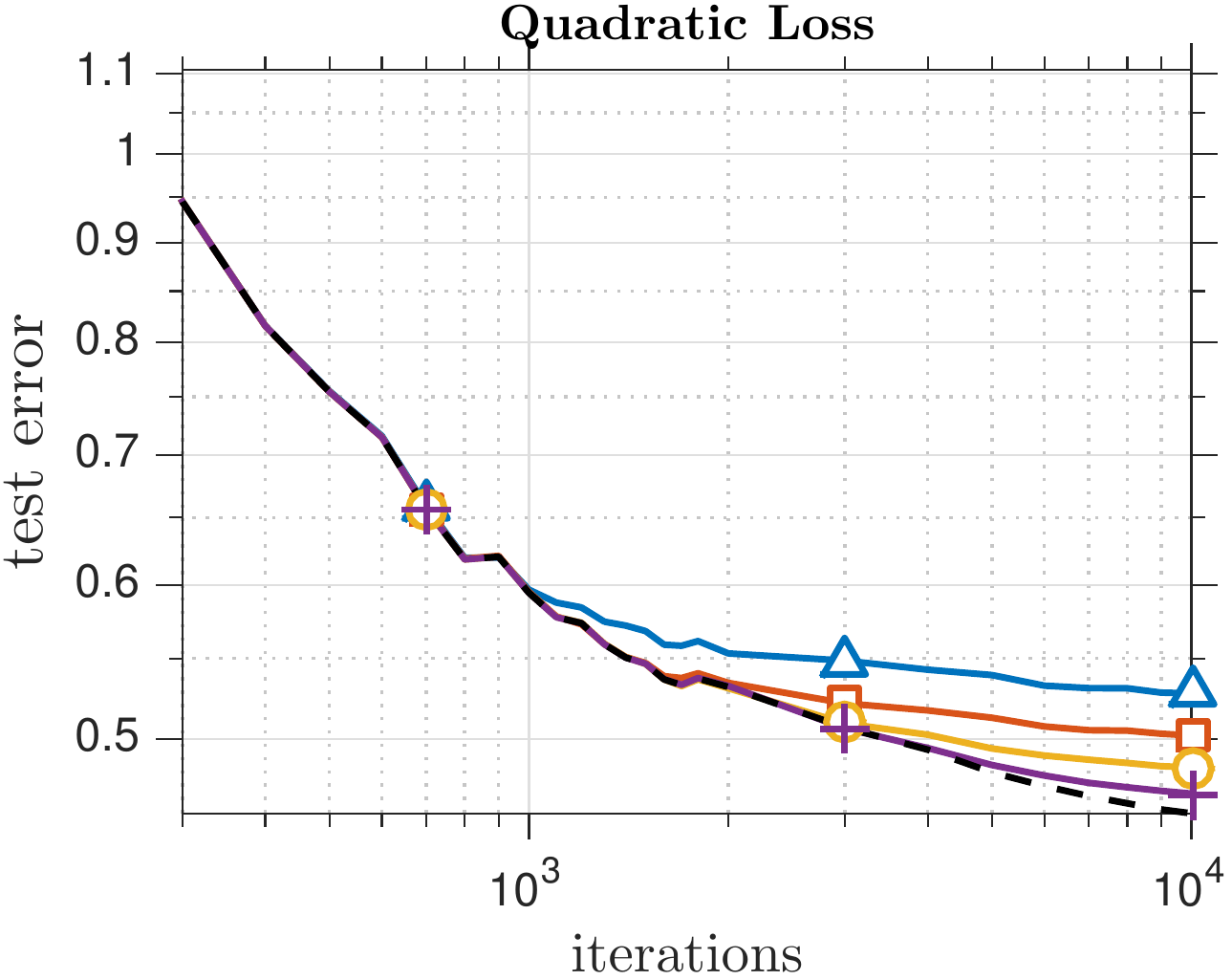}
\end{subfigure}
~
\begin{subfigure}{0.32\textwidth}
\includegraphics[width = \textwidth]{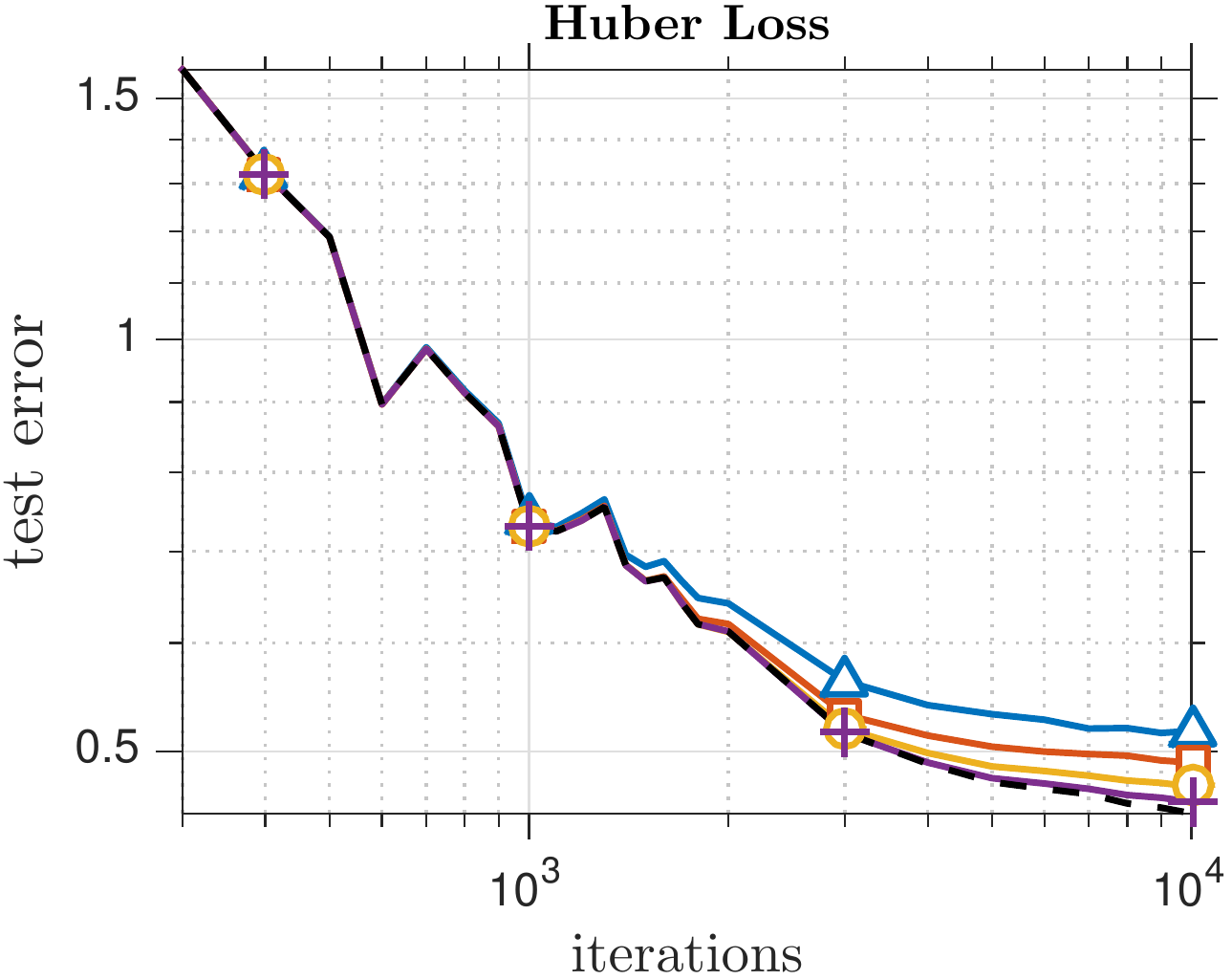}
\end{subfigure}
~
\begin{subfigure}{0.32\textwidth}
\includegraphics[width = \textwidth]{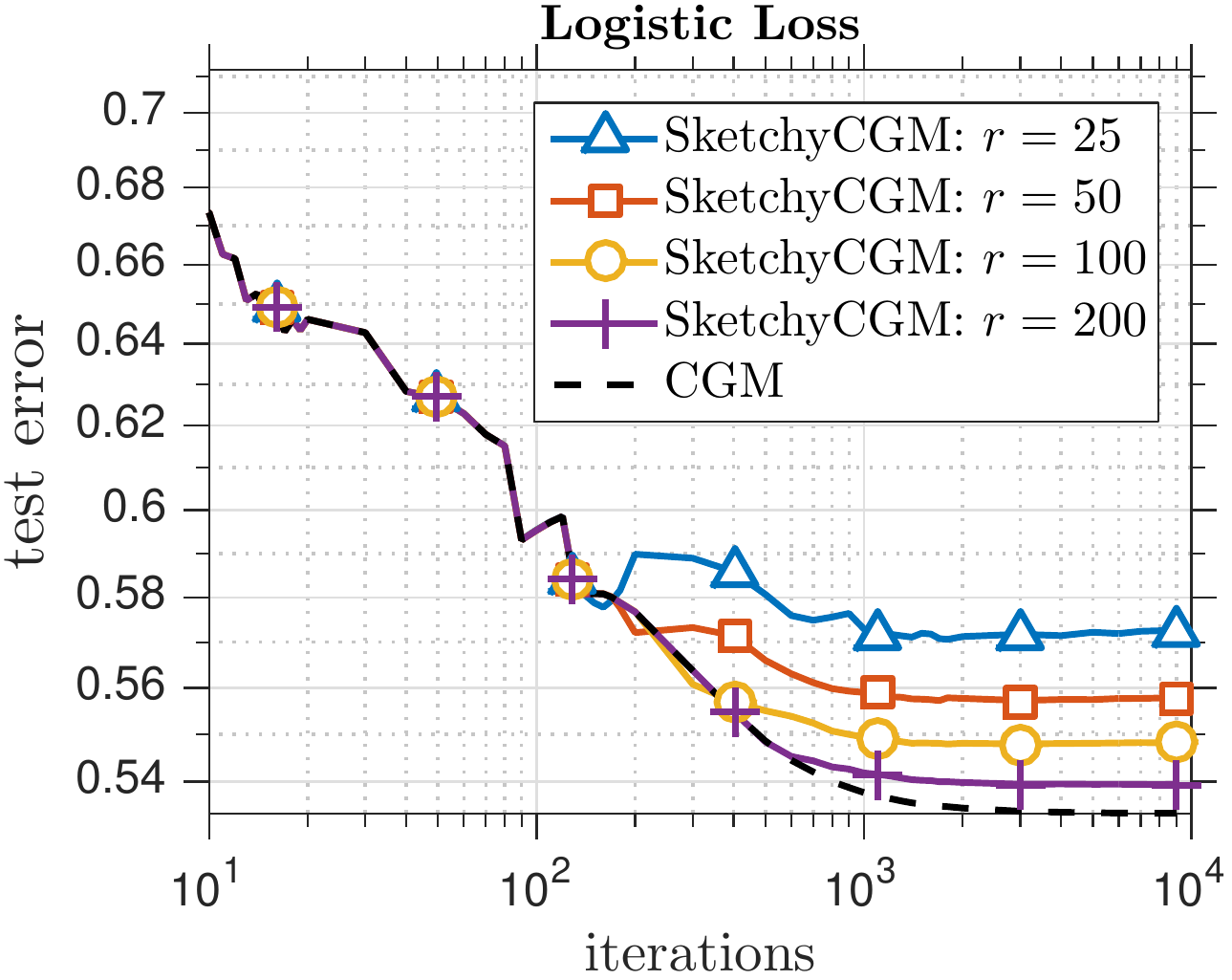}
\end{subfigure}
\caption{\label{f-movielens-10m}\textsl{MovieLens 10M Dataset.}
Convergence of test error for CGM and SketchyCGM solutions to \eqref{eqn:model-prob} with three loss functions.
The parameter $\alpha$ is chosen by cross-validation on the final CGM solution: $\alpha = 150,\!000$ for quadratic and logistic losses, and $175,\!000$ for Huber loss.  See \S\ref{sec:mc-app} for details.}
\end{figure*}

\subsection{Phase retrieval problems}
\label{app:PR}

Section \ref{sec:numerics} showcases the empirical performance of SketchyCGM
on a family of phase retrieval problems.  This section provides details of
these experiments, as well as some related examples.

\subsubsection{Overview}

The setup is the same as in \S\ref{sec:phase-vignette}.
Let $\vct{x}_{\natural} \in \C^n$ be a vector, and suppose
we acquire measurements
\begin{equation} \label{eqn:phase-retrieval-app}
b_i = \abs{\ip{ \vct{a}_i }{ \vct{x}_{\natural} }}^2 + \xi_i
\quad\text{for $i = 1, \dots, d$.}
\end{equation}
We can modify the measurement vectors $\vct{a}_i \in \C^n$
to obtain a range of problems.
We can also adjust the distribution of the noise $\xi_i$.

To model the measurement process~\eqref{eqn:phase-retrieval-app},
it is convenient to form the matrix $\mtx{A} \in \C^{d \times n}$
whose rows are the measurement vectors $\vct{a}_i^*$.
Then define a linear map $\mathcal{A} : \C^{n \times n} \to \C^d$
and its adjoint $\mathcal{A}^* : \C^d \to \C^{n \times n}$ via
\begin{equation} \label{eqn:phase-lin-map}
\begin{aligned}
\mathcal{A}\mtx{X} \ &= \ \diag( \mtx{A} \mtx{X} \mtx{A}^* ); \\
\mathcal{A}^* \vct{z} \ &= \ \mtx{A}^* \diag^*(\vct{z}) \mtx{A}.
\end{aligned}
\end{equation}
The map $\diag : \C^{d \times d} \to \C^d$ extracts the diagonal
of a matrix; $\diag^* : \C^d \to \C^{d\times d}$
maps a vector into a diagonal matrix.
When $\mtx{X}$ is psd, note that $\mathcal{A}\mtx{X} \in \R_+^d$.

We instantiate the convex optimization template~\eqref{eqn:model-prob-psd}
with the linear map~\eqref{eqn:phase-lin-map}
and the objective function
$$
f_{\ell}(\vct{z}) \ = \ \sum\nolimits_{i=1}^d \psi_{\ell}(z_i, b_i).
$$
Here, the parameter $\ell \in \{ \mathrm{gauss}, \mathrm{poisson} \}$. 
Following~\cite[Sec.~II]{YHC15:Scalable-Convex},
we usually set $\alpha = d^{-1} \sum_{i=1}^d b_i$.
We approximate the true vector $\vct{x}_{\natural}$ by means of a
maximum eigenvector $\vct{x}_{\star}$ of a solution $\mtx{X}_{\star}$
to~\eqref{eqn:model-prob-psd}.

\begin{table*}[t]
\centering
\caption{\textsl{Memory usage.}  Approximate memory usage (in bytes) for five convex solvers 
applied to the convex phase retrieval problem~\eqref{eqn:phase-convex},
as shown in Figure~\ref{fig:SpaceConv}.
The dash --- indicates that an algorithm ran out of memory.}
\vspace{0.5pc}

\small
\begin{tabular}{l llllllll}
Signal length $(n)$
	& $\phantom{0.00 \cdot {}}10$
	& $\phantom{0.00 \cdot {}}10^2$
	& $\phantom{0.00 \cdot {}}10^3$
	& $\phantom{0.00 \cdot {}}10^4$
	& $\phantom{0.00 \cdot {}}10^5$
	& $\phantom{0.00 \cdot {}}10^6$
\rule[-1ex]{0pt}{0pt} \\
\hline
\rule{0pt}{2.6ex} \!\!\!\!
AT & $2.82 \cdot 10^4$ & $1.32 \cdot 10^6$ & $1.21 \cdot 10^8$ & $1.20 \cdot 10^{10}$ & --- & --- \\
PGM & $2.90 \cdot 10^4$ & $1.07 \cdot 10^6$ & $9.70 \cdot 10^7$ & $9.61 \cdot 10^9$ & --- & --- \\
CGM & $6.36  \cdot 10^3$ & $1.99 \cdot 10^5$ & $1.63 \cdot 10^7$ & $1.60 \cdot 10^9$ & --- & ---  \\
ThinCGM & $6.36  \cdot 10^3$ & $1.08 \cdot 10^5$ & $3.17 \cdot 10^6$ & $1.42 \cdot 10^8$ & $1.53 \cdot 10^9$ & --- \\
SketchyCGM & $1.06  \cdot 10^4$ & $9.06 \cdot 10^4$ & $8.90 \cdot 10^5$ & $8.88 \cdot 10^6$ & $8.88 \cdot 10^7$ & $8.88 \cdot 10^8$ \\
\hline
\vspace{1.5pc}
\end{tabular}
\label{tab:SpaceEff}
\end{table*}

\begin{figure*}[!t]
    \centering
    \begin{subfigure}{0.3\textwidth}
        \centering
        \includegraphics[width=\textwidth]{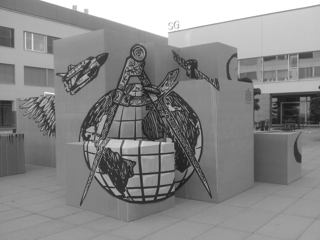}
        \caption{Original image ($240 \times 320$)}
    \end{subfigure}
    \quad
    \begin{subfigure}{0.3\textwidth}
        \centering
        \includegraphics[width=\textwidth]{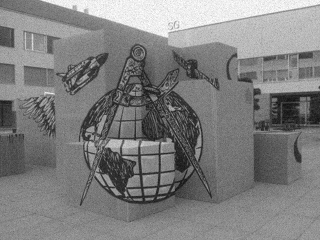}
        \caption{Gaussian loss: PSNR 26.89 dB}
    \end{subfigure}
    \quad
    \begin{subfigure}{0.3\textwidth}
        \centering
        \includegraphics[width=\textwidth]{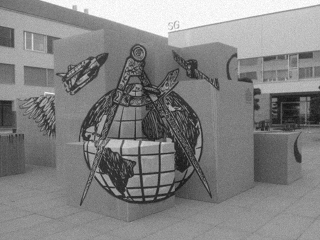}
        \caption{Poisson loss: PSNR 32.12 dB}
    \end{subfigure}
\caption{\textsl{Gaussian and Poisson phase retrieval under Poisson noise.}
See \S\ref{sec:poisson} for details.}
\label{fig:poisson}
\end{figure*}

\subsubsection{Synthetic phase retrieval}
\label{sec:synthetic-app}

In \S\ref{sec:synthetic-PR}, we considered a measurement model based on random coded diffraction patterns.
This is a synthetic setup inspired by an imaging system where
one can modulate the image before diffraction occurs~\cite{CLS15:Phase-Retrieval}.

For this example, the matrix $\mtx{A} \in \C^{d \times n}$ appearing in~\eqref{eqn:phase-lin-map}
takes the form
\begin{equation} \label{eqn:coded-diff}
\mtx{A} = \begin{bmatrix}
	\mtx{F}_n & \mtx{0}  &\cdots & \mtx{0} \\
	\mtx{0} &  \mtx{F}_n &\cdots & \mtx{0} \\
	\vdots &  \vdots  &\ddots & \vdots \\
	\mtx{0} &  \mtx{0} &\cdots & \mtx{F}_n
	\end{bmatrix}
	\begin{bmatrix}
	\mtx{D}_1 \\
	\mtx{D}_2 \\
	\vdots \\
	\mtx{D}_s \\
	\end{bmatrix}.
\end{equation}
In this expression, $\mtx{F}_n \in \C^{n\times n}$ is the discrete Fourier transform (DFT) matrix,
and $\mtx{D}_i \in \C^{n \times n}$ are diagonal matrices that describe modulating waveforms.
The parameter $s$ represents the number of $n$-dimensional views of the target vector
$\vct{x}_{\natural} \in \C^n$ that we acquire.  The total number of measurements $d = sn$.

We generate each diagonal entry of each matrix $\mtx{D}_i$ independently at random.
Each one is the product of two independent random variables $U_1$ and $U_2$,
where $U_1$ is chosen uniformly from $\{1, \mathrm{i},-1,-\mathrm{i}\}$
and $U_2$ is drawn from $\{\sqrt{2}/2, \sqrt{3}\}$ with probabilities $0.8$ and $0.2$.

In this setting, we can represent the linear map $\mathcal{A}$ in~\eqref{eqn:phase-lin-map}
using $3d$ bits, and we can apply it efficiently using the FFT algorithm.

For the scaling experiments in \S\ref{sec:synthetic-PR}, the measurements take the
form~\eqref{eqn:phase-retrieval-app} where the $\vct{a}_i^*$
are the rows of \eqref{eqn:coded-diff}.  We solve~\eqref{eqn:model-prob-psd}
with the loss $f_{\mathrm{gauss}}$, the linear map~\eqref{eqn:phase-lin-map}--\eqref{eqn:coded-diff},
and $\alpha$ set to the average of the data $b_i$.
The remaining details appear in \S\ref{sec:synthetic-PR}.
Table~\ref{tab:SpaceEff} summarizes
the storage costs for solving this type of synthetic phase retrieval problem
with five different convex optimization algorithms.

\begin{figure*}[!t]
    \centering
    \begin{subfigure}{0.195\textwidth}
        \centering
        \includegraphics[height=3.2cm]{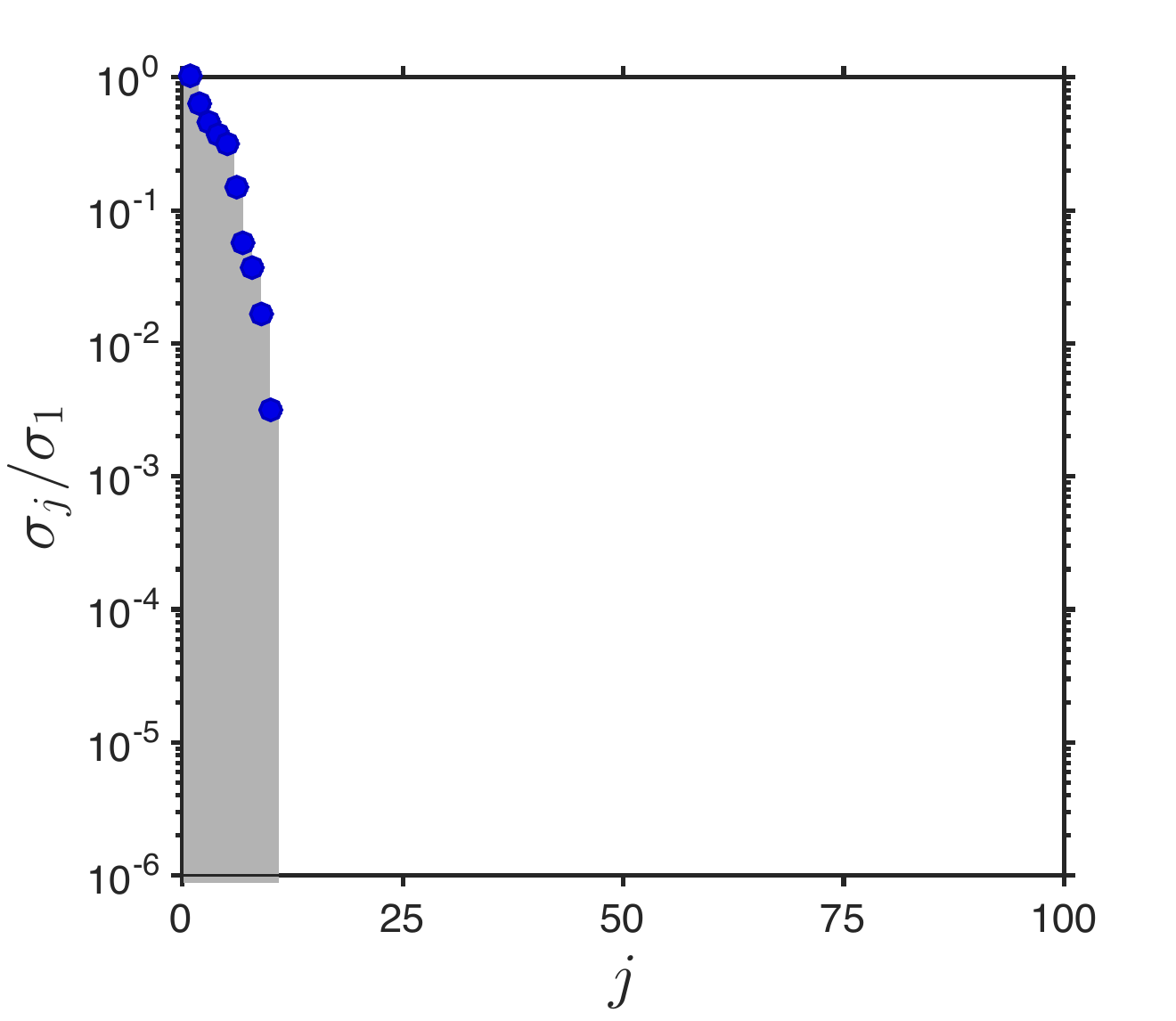}
        \caption*{iteration $10$}
    \end{subfigure}
    \begin{subfigure}{0.195\textwidth}
        \centering
        \includegraphics[height=3.2cm]{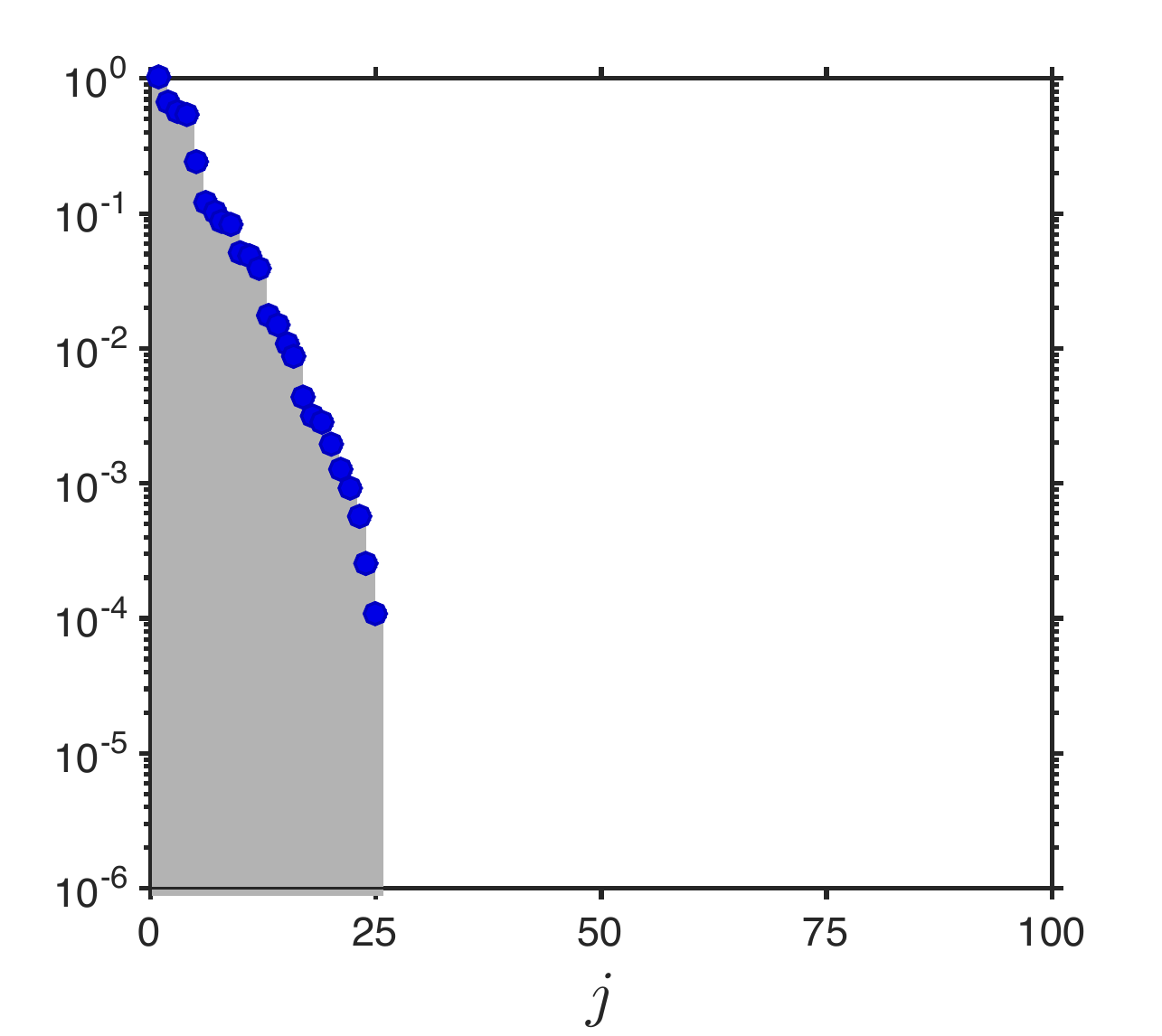}
        \caption*{iteration $25$}
    \end{subfigure}
    \begin{subfigure}{0.195\textwidth}
        \centering
        \includegraphics[height=3.2cm]{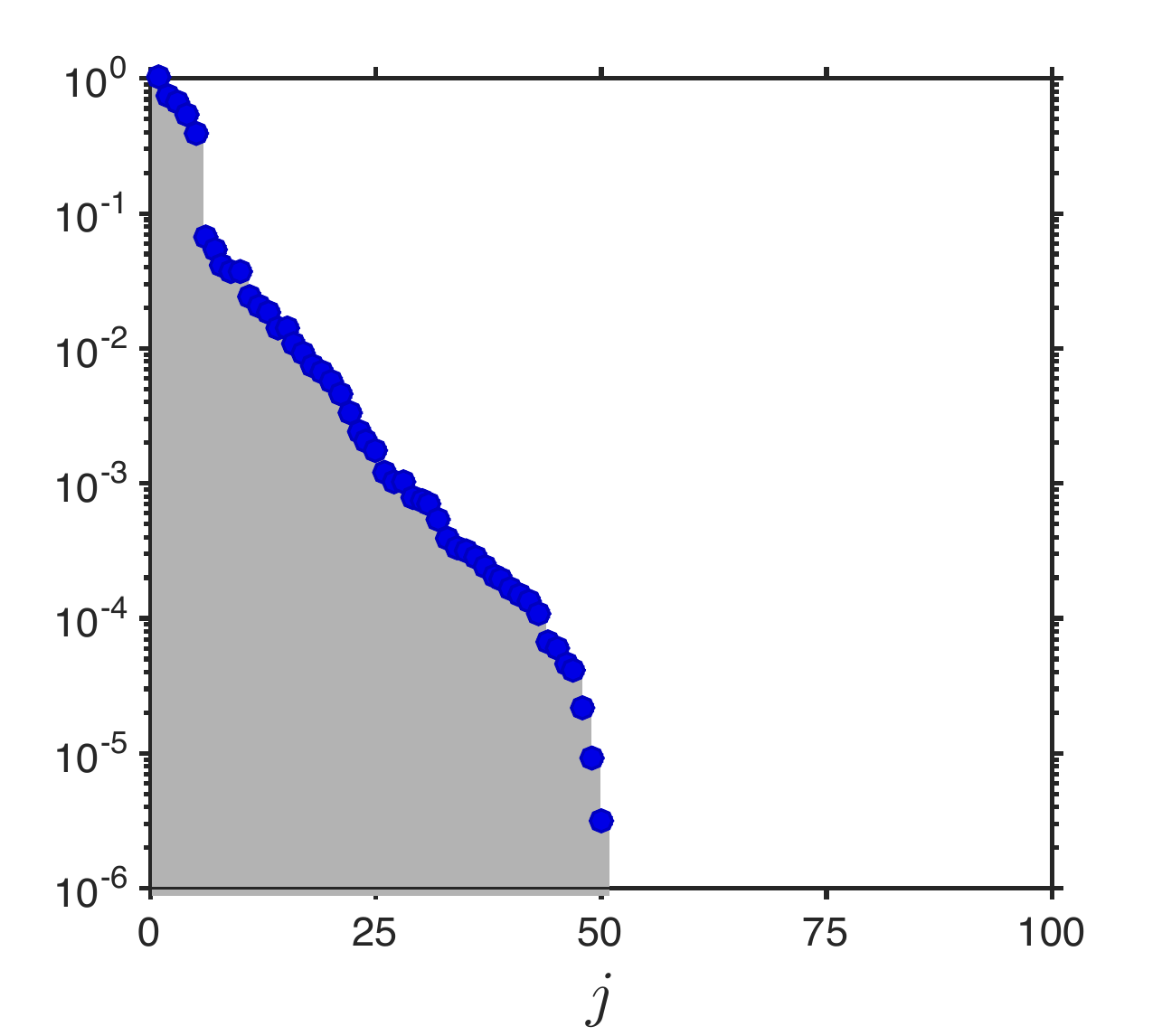}
        \caption*{iteration $50$}
    \end{subfigure}
    \begin{subfigure}{0.195\textwidth}
        \centering
        \includegraphics[height=3.2cm]{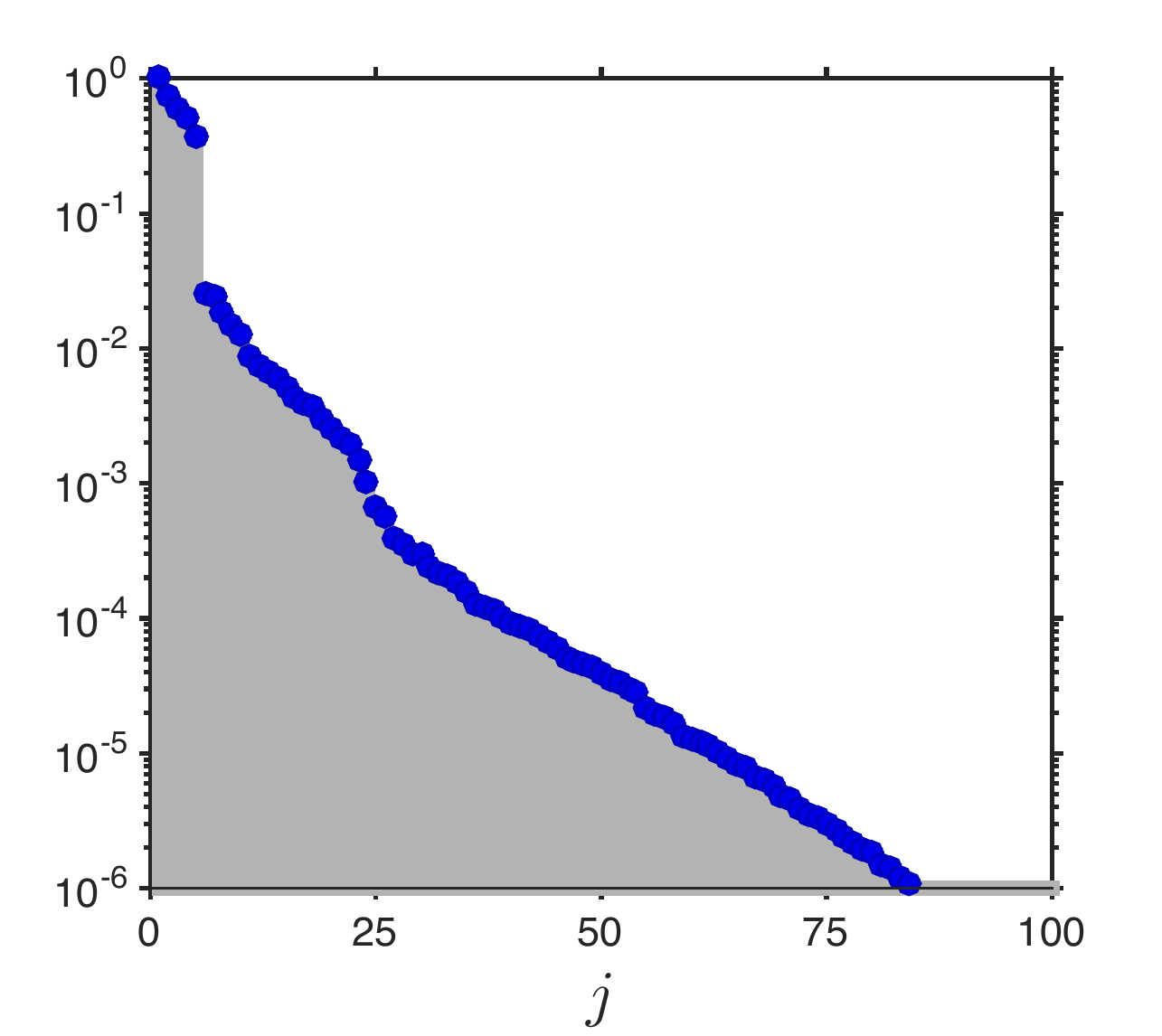}
        \caption*{iteration $100$}
    \end{subfigure}
    \begin{subfigure}{0.195\textwidth}
        \centering
        \includegraphics[height=3.2cm]{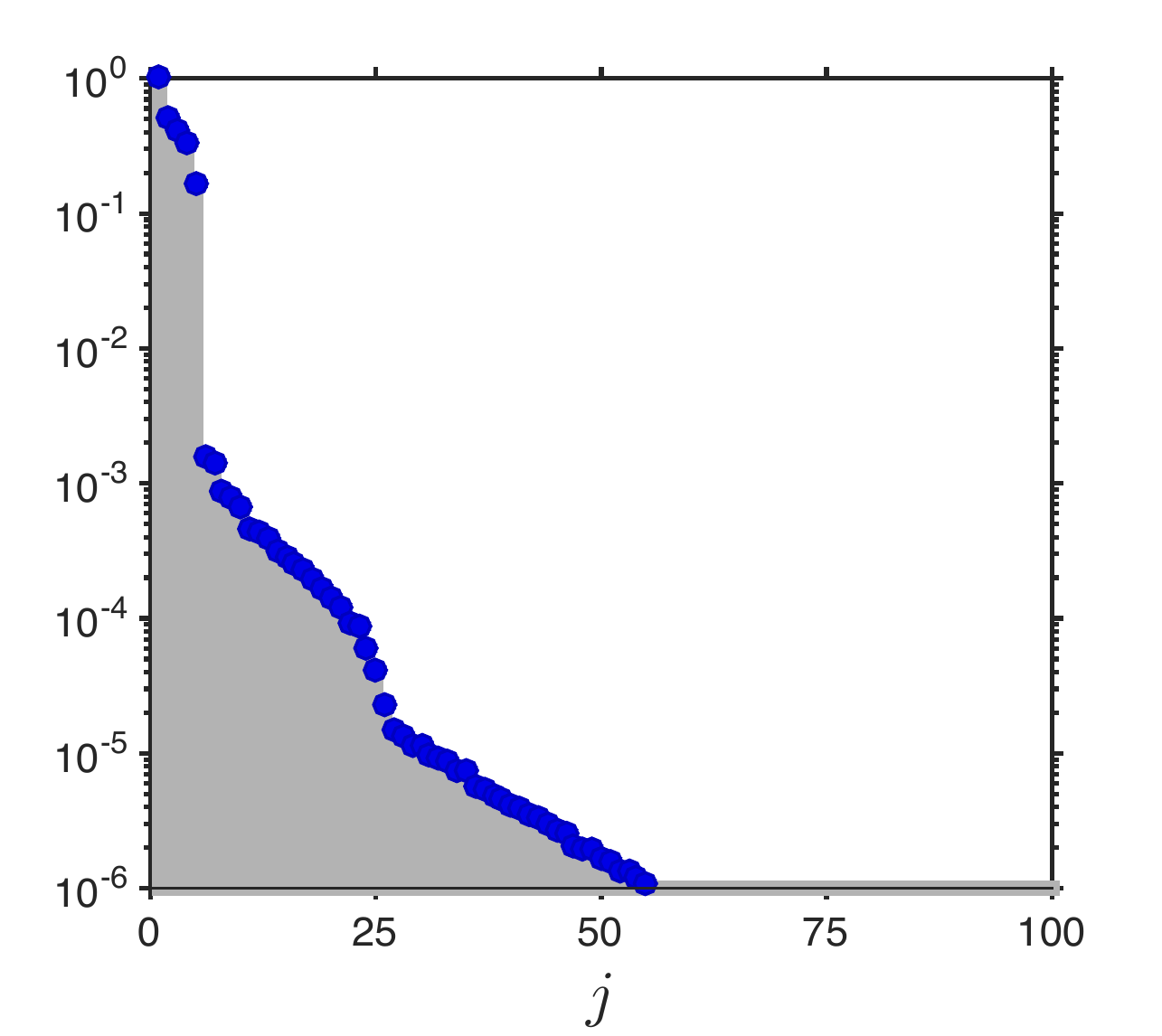}
        \caption*{iteration $500$}
    \end{subfigure}
\caption{\textsl{Spectrum of the CGM Iterates.}
 CGM is applied to the phase retrieval problem~\eqref{eqn:phase-convex}
 with the blood cell image data described in~\S\ref{sec:fourier-ptychography-problem}.
 These plots show the relative singular value spectrum $\sigma_j/\sigma_1$ for five
 specific iterates.  The solution to~\eqref{eqn:phase-convex} appears to have rank five.
 See \S\ref{sec:FP-app} for discussion.}
\label{fig:evolution-spectrum}
\end{figure*}

\begin{figure*}[t]
\begin{center}
\begin{subfigure}{0.48\linewidth}
\includegraphics[height=6cm]{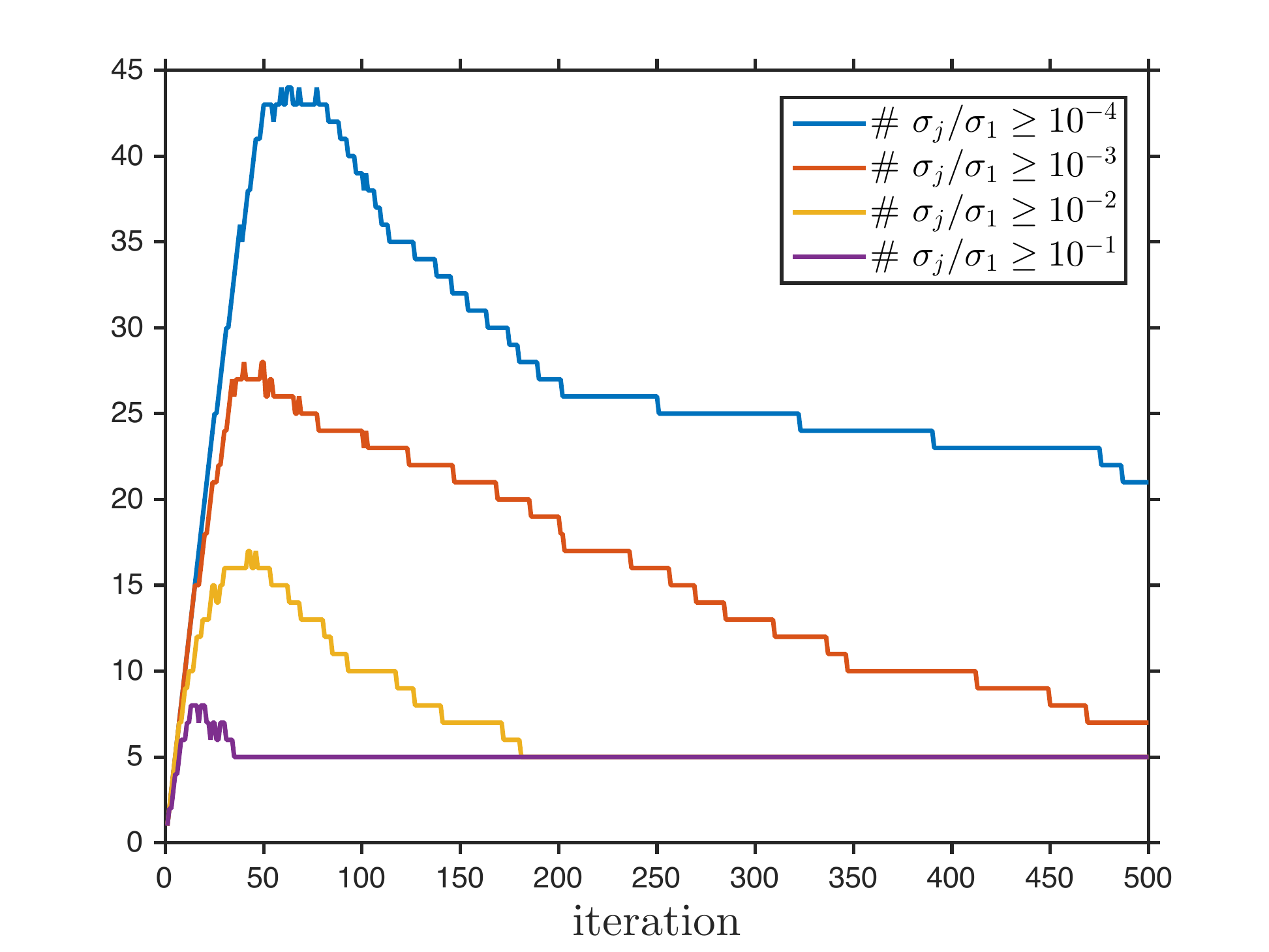}
\caption{The $\eps$-rank of the CGM iterates}
\end{subfigure}
\begin{subfigure}{0.48\linewidth}
\includegraphics[height=6cm]{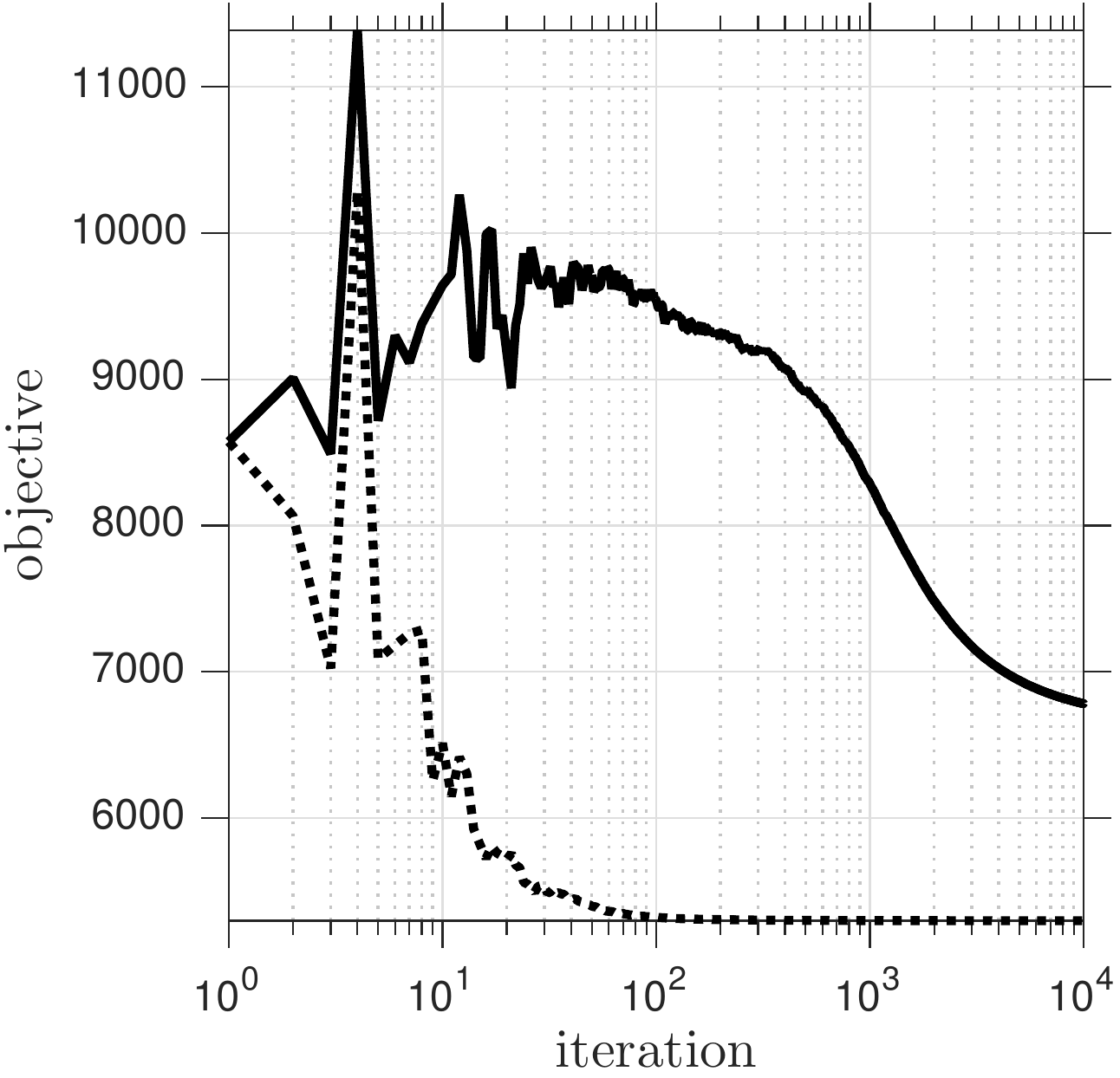}
\caption{SketchyCGM objective value}
\end{subfigure}
\caption{\textsl{Convergence of CGM and SketchyCGM.}
CGM and SketchyCGM are applied to the phase retrieval problem~\eqref{eqn:phase-convex}
 with the blood cell image data described in~\S\ref{sec:fourier-ptychography-problem}.
\textsc{(a)} This panel displays the evolution of the $\eps$-rank of the CGM iterates
for several choices of $\eps$.
\textsc{(b)} This panel shows the objective values achieved by the CGM iterates $\mtx{X}_t$
{(dashed)} and the rank-one SketchyCGM estimates $\hat{\mtx{X}}_t$ {(solid)}.
There is a gap because the solution to~\eqref{eqn:phase-convex} appears to have rank five;
see \textsc{(a)}.}
\label{fig:fourier-ptychography-convergence}
\end{center}
\end{figure*}

\begin{figure*}[t]
    \centering
    \begin{subfigure}{0.23\textwidth}
        \centering
        \includegraphics[width=\textwidth]{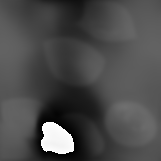}
        \caption*{iteration $10$}
    \end{subfigure}
    ~
    \begin{subfigure}{0.23\textwidth}
        \centering
        \includegraphics[width=\textwidth]{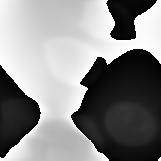}
        \caption*{iteration $100$}
    \end{subfigure}
    ~
    \begin{subfigure}{0.23\textwidth}
        \centering
        \includegraphics[width=\textwidth]{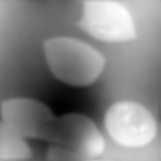}
        \caption*{iteration $1,\!000$}
    \end{subfigure}
    ~
    \begin{subfigure}{0.23\textwidth}
        \centering
        \includegraphics[width=\textwidth]{figs/SketchyCGM10000iter.png}
        \caption*{iteration $10,\!000$}
    \end{subfigure}
\caption{\textsl{Phase Reconstruction from Fourier Ptychographic Data.}
SketchyCGM is applied to the phase retrieval problem \eqref{eqn:phase-convex}
with the blood cell image data described in \S\ref{sec:fourier-ptychography-problem}.
These panels show the evolution of the phase reconstruction as a function of the
iteration.  Brightness indicates the complex phase; only relative differences are meaningful.
The image has diagnostic quality after 1,000 iterations, but it continues to sharpen.} 
\label{fig:reconstruction-by-iterations}
\end{figure*}

\subsubsection{Synthetic Poisson phase retrieval}
\label{sec:poisson}

In many imaging systems, a Poisson noise model is more appropriate than
a Gaussian noise model.  Let us demonstrate that the SketchyCGM
algorithm can solve synthetic phase retrieval problems with the
loss $f_{\mathrm{poisson}}$.  This work highlights the importance
of adapting the loss function to the noise distribution.

The setup is similar to~\S\ref{sec:synthetic-app}.
Fix a small image $\vct{x}_{\natural} \in \C^n$
with $n = 240 \times 320 = 76,\!800$ pixels.
Acquire $d = 20 n$ measurements of the form~\eqref{eqn:phase-retrieval-app}
using the coded diffraction model~\eqref{eqn:phase-lin-map}--\eqref{eqn:coded-diff}.
Each realization $\xi_i \in \R^d$ of the noise is drawn iid from a $\textsc{Poisson}$
distribution, whose mean is chosen so that the SNR of the measurements is 20 dB.

We formulate the phase retrieval problem using the template~\eqref{eqn:model-prob-psd}
with the loss $f_{\mathrm{poisson}}$,
the linear map~\eqref{eqn:phase-lin-map}--\eqref{eqn:coded-diff},
and with $\alpha$ set to the average of the measurements $b_i$.

To solve this problem via SketchyCGM,
it is necessary to make some small modifications~\cite{OLY+16:Frank-Wolfe-Works}.
We initialize the algorithm with the dual vector $\vct{z}_0 = d^{-1/2} \vct{1}$
and set the learning rate $\eta_t = 2/(3+t)$.  The rank parameter
$r = 1$, and the algorithm runs for 100 iterations.

Figure~\ref{fig:poisson} displays the results of this computation.
We also compare the output with a reconstruction obtained 
by applying the unmodified version of SketchyCGM to
solve~\eqref{eqn:model-prob-psd} with the (mismatched) loss
$f_{\mathrm{gauss}}$, the same linear map, and the same value of $\alpha$.
As expected, the Poisson formulation performs better.

\subsubsection{Fourier ptychography}
\label{sec:FP-app}

In \S\ref{sec:fourier-ptychography-problem}, we discussed a real-world
phase retrieval problem arising from Fourier ptychography~\cite{HCO+15:Solving-Ptychography}.
Let us explain the mathematical model for this problem
and present some additional numerical work.

For Fourier ptychography, the matrix $\mtx{A} \in \C^{d \times n}$ 
appearing in~\eqref{eqn:phase-lin-map} has the following structure.
\begin{equation} \label{eqn:fp-meas}
\mtx{A} = \begin{bmatrix}
	\mtx{F}_q^* & \mtx{0}  &\cdots & \mtx{0} \\
	\mtx{0} &  \mtx{F}_q^* &\cdots & \mtx{0} \\
	\vdots &  \vdots  &\ddots & \vdots \\
	\mtx{0} &  \mtx{0} &\cdots & \mtx{F}_q^*
	\end{bmatrix}
	\begin{bmatrix}
	\mtx{D}_1 \\
	\mtx{D}_2 \\
	\vdots \\
	\mtx{D}_s \\
	\end{bmatrix}
	\mtx{F}_n.
\end{equation}
In this expression, $\mtx{F}_n \in \C^{n \times n}$
is the 2D discrete Fourier transform,
and $\mtx{F}_q^* \in \C^{q \times q}$
is a low-dimensional 2D discrete Fourier transform.
The sparse matrices $\mtx{D}_{i} \in \C^{q\times n}$
describe bandpass filters;
each column of $\mtx{D}_i$ has at most one nonzero entry.
The number of measurements $d = sq$.
See~\cite{HCO+15:Solving-Ptychography} for details
about the physical setup and the mathematical model.

Section~\ref{sec:fourier-ptychography-problem} describes
a specific instance of Fourier ptychography imaging,
applied to a slide containing red blood cells.
To perform phase retrieval, we use the optimization problem~\eqref{eqn:model-prob-psd}
with loss $f_{\mathrm{gauss}}$ and with $\alpha = 1,\!400$.
The measurement vectors $\vct{a}_i^*$ are the rows of \eqref{eqn:fp-meas}.
We apply several algorithms, including a moderate number of iterations of CGM,
SketchyCGM with rank parameter $r = 1$,
the Burer--Monteiro method~\cite{BM03:Nonlinear-Programming,HCO+15:Solving-Ptychography},
and Wirtinger flow~\cite{CLS15:Phase-Retrieval}.

Figure~\ref{fig:evolution-spectrum}
and~\ref{fig:fourier-ptychography-convergence}(\textsc{a})
provide information about the spectrum of the CGM iterates.
For several values of $\eps$, we observe that the $\eps$-rank%
\footnote{The $\eps$-rank of a matrix is the number of singular values
that exceed $\eps \sigma_1$, where $\sigma_j$ is the $j$th largest singular value.}
of the iterates becomes large before declining to the value five.
This type of behavior is typical for CGM, and it scuttles CGM variants that
try to control the rank of the iterates directly.

Ideally, the solution $\mtx{X}_{\star}$ to the convex
formulation~\eqref{eqn:phase-convex} of a phase retrieval
problem has rank one.  But these computations suggest
that, for the blood cell data, the solution actually has rank five.
The increase in rank is due to nonidealities in the imaging system,
such as the spatial incoherence of the light source.  In essence,
the measurements capture a superposition of several slightly different
images.  Regardless, the convex model~\eqref{eqn:phase-convex}
is still effective, and a top eigenvector of the solution $\mtx{X}_{\star}$
still provides a good approximation to the image~\cite{HCO+15:Solving-Ptychography}.

Figure~\ref{fig:fourier-ptychography-convergence}(\textsc{b}) charts
the objective value $f(\mathcal{A}\hat{\mtx{X}}_t)$ attained by
the SketchyCGM iterates. 
We can implicitly compute the objective value $f(\mathcal{A}\mtx{X}_t)$ for the CGM iterate $\mtx{X}_t$
using the loop invariant~\eqref{eqn:CGM-sketch-loop}.  Note that
$\mtx{X}_t$ achieves a much smaller objective value than the
rank-one approximation $\hat{\mtx{X}}_t$ produced by SketchyCGM.
The discrepancy is due to the fact that the solution to
the optimization problem has approximate rank five.

Figure~\ref{fig:reconstruction-by-iterations} displays
snapshots of the SketchyCGM iterates as the algorithm proceeds.
We see that SketchyCGM already achieve a diagnostic quality
image after $1,\!000$ iterations, but the algorithm continues
to resolve the image as it runs.

Last, Figure~\ref{fig:ptych-grad} shows the phase gradient
of the solution to~\eqref{eqn:phase-convex} obtained with three different algorithms;
these plots provide an alternative view of Figure~\ref{fig:ptych}.
Roughly, the phase gradient indicates the change in the thickness of the sample
at a given location.  Therefore, absolute changes in the value of
the phase gradient are meaningful.
Note the unphysical oscillations in the reconstruction via
Burer--Monteiro~\cite{BM03:Nonlinear-Programming,HCO+15:Solving-Ptychography}.
The reconstruction via Wirtinger Flow~\cite{CLS15:Phase-Retrieval}
contains no information at all.

\begin{figure*}[!t]
    \centering
    \begin{subfigure}{0.3\textwidth}
        \centering
        \includegraphics[width=\textwidth]{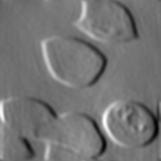}
    \end{subfigure}
    ~
    \begin{subfigure}{0.3\textwidth}
        \centering
        \includegraphics[width=\textwidth]{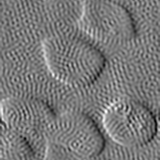}
    \end{subfigure}
    ~
    \begin{subfigure}{0.3\textwidth}
        \centering
        \includegraphics[width=\textwidth]{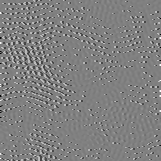}
    \end{subfigure}

	\vspace{0.5pc}

    \begin{subfigure}{0.3\textwidth}
        \centering
        \includegraphics[width=\textwidth]{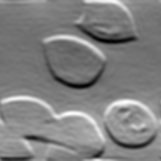}
        \caption{SketchyCGM}
    \end{subfigure}
    ~
    \begin{subfigure}{0.3\textwidth}
        \centering
        \includegraphics[width=\textwidth]{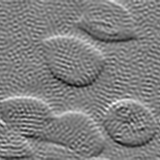}
        \caption{Burer--Monteiro}
    \end{subfigure}
    ~
    \begin{subfigure}{0.3\textwidth}
        \centering
        \includegraphics[width=\textwidth]{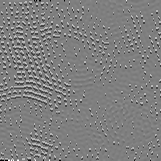}
        \caption{Wirtinger Flow}
    \end{subfigure}
\caption{\textsl{Phase Gradients for Fourier Ptychographic Imaging.}
Three algorithms are applied to the phase retrieval problem \eqref{eqn:phase-convex}
with the blood cell image data described in \S\ref{sec:fourier-ptychography-problem}.
	[\textsl{top}] The horizontal differences of the phase maps presented in Figure~\ref{fig:ptych}.
	[\textsl{bottom}] The vertical differences of the phase maps.  Brightness roughly
	corresponds to a change in thickness of the sample.}

 \label{fig:ptych-grad}
 \end{figure*}

\bibliographystyle{plainnat}
{\small \bibliography{ctuy16-small-bib}}

\end{document}